\documentclass[10pt,reqno]{amsart}

\usepackage{amssymb,latexsym,mathrsfs,amsmath}
\usepackage{amsthm}
\usepackage{graphicx}
\usepackage{tikz}
\usetikzlibrary{arrows}
\usepackage{caption}
\usepackage{subcaption}
\usepackage{etex}
\usepackage[all,2cell]{xy}
\SelectTips{eu}{12}

\newtheorem{lemma}{Lemma}[section]

\newtheorem{theorem}[lemma]{Theorem}

\theoremstyle{definition}
\newtheorem{remark}[lemma]{Remark}
\newtheorem{definition}[lemma]{Definition}
\newtheorem{example}[lemma]{Example}

\DeclareMathOperator{\id}{id}
\DeclareMathOperator{\Kh}{Kh}
\DeclareMathOperator{\Mor}{Mor}
\DeclareMathOperator{\Sq}{Sq}

\newcommand{\Z}{\mathbb{Z}}

\newcommand{\Q}{\mathbb{Q}}
\newcommand{\R}{\mathbb{R}}

\newcommand{\F}{\mathbb{F}}
\newcommand{\BN}{H_{\mathrm{BN}}}
\newcommand{\cC}{\mathfrak{C}}
\newcommand{\ind}{\mathrm{ind}\,}

\newcommand{\Sqo}[1] {
\Sq^1_{\mathrm{#1}}}

\newcommand{\Kho}[1] {
\Kh_{\mathrm{o}}^{#1}
}

\newcommand{\rKh}[1] {
\widetilde{\Kh}_{\mathrm{o}}\!\!\!^{#1}
}

\newcommand{\dChr}{\mathrm{Chr}_\bullet}
\newcommand{\dChrq}{\mathrm{Chr}_{\bullet/}}
\newcommand{\dChrqu}[1]{\mathrm{Chr}_{\bullet/#1}}

\newcommand{\crossing}[4] {
\draw[#4](#1,#2) -- (#1+#3,#2+#3);
\draw[#4](#1+#3, #2) -- (#1+0.6 * #3, #2+0.4 * #3);
\draw[#4](#1,#2+#3) -- (#1+0.4 * #3, #2 + 0.6 * #3);
}

\newcommand{\smoothingup}[4]{
\draw[#4] (#1,#2) -- (#1+0.2 * #3, #2+0.2 * #3) to [out = 45, in = 315] (#1+0.2 * #3, #2+0.8 * #3) -- (#1, #2 + #3);
\draw[#4] (#1+#3,#2) -- (#1 + 0.8 * #3, #2 + 0.2 * #3) to [out = 135, in = 225] (#1 + 0.8 * #3, #2 + 0.8 * #3) -- (#1+#3,#2+#3);
}

\newcommand{\smoothinglr}[4]{
\draw[#4](#1,#2) -- (#1+0.2 * #3, #2+0.2 * #3) to [out = 45, in = 135] (#1 + 0.8 * #3, #2 + 0.2 * #3) -- (#1+#3,#2);
\draw[#4](#1,#2+#3) --  (#1+0.2 * #3, #2+0.8 * #3) to [out = 315, in = 225]  (#1+0.8 * #3, #2+0.8 * #3) -- (#1+#3,#2+#3);
}

\newcommand{\widesmoothingup}[4]{
\draw[#4] (#1,#2) -- (#1+0.2 * #3, #2+0.2 * #3) to [out = 45, in = 315] (#1+0.2 * #3, #2+0.8 * #3) -- (#1, #2 + #3);
\draw[#4] (#1+2*#3,#2) -- (#1 + 1.8 * #3, #2 + 0.2 * #3) to [out = 135, in = 225] (#1 + 1.8 * #3, #2 + 0.8 * #3) -- (#1+2*#3,#2+#3);
}

\newcommand{\arroweast}[3] {
\draw[->] (#1,#2) -- (#1+#3, #2);
}

\newcommand{\dotting}[3] {
\draw[very thick] (#1 - 0.7*#3, #2 - 0.7*#3) -- (#1 - 0.4*#3, #2 - 0.4*#3) to [out = 45, in = 135] ((#1 + 0.4*#3, #2 - 0.4*#3) -- (#1 + 0.7*#3, #2 - 0.7*#3);
\node[scale = 2*#3] at (#1, #2 - 0.25 * #3) {$\bullet$};
}

\newcommand{\saddlenorth}[3] {
\draw[very thick] (#1 - 0.7*#3, #2 - 0.7*#3) -- (#1 - 0.4*#3, #2 - 0.4*#3) to [out = 45, in = 135] ((#1 + 0.4*#3, #2 - 0.4*#3) -- (#1 + 0.7*#3, #2 - 0.7*#3);
\draw[very thick] (#1 - 0.7*#3, #2 + 0.7*#3) -- (#1 - 0.4*#3, #2 + 0.4*#3) to [out = 315, in = 225] ((#1 + 0.4*#3, #2 + 0.4*#3) -- (#1 + 0.7*#3, #2 + 0.7*#3);
\draw[->] (#1, #2 - 0.26 * #3) -- (#1, #2 + 0.26 * #3);
}

\newcommand{\bowl}[4]{
\shade[ball color = gray!40, opacity = 0.3] (#1+0.8*#3,#2) arc (0:180:0.4*#3 and -0.2*#3) -- (#1,#2-#4*#3) arc (180:360:0.4*#3 and 0.4*#3) -- (#1+0.8*#3,#2);
\draw (#1+0.8*#3,#2) arc (0:180:0.4*#3 and -0.2*#3) -- (#1,#2-#4*#3) arc (180:360:0.4*#3 and 0.4*#3) -- (#1+0.8*#3,#2);
\draw (#1,#2) arc (180:360:0.4*#3 and -0.2*#3);
\shade[ball color = gray!40, opacity = 0.15] (#1,#2) arc (180:360:0.4*#3 and -0.2*#3) arc (0:180:0.4*#3 and -0.2*#3);
}

\newcommand{\bowld}[3]{
\shade[ball color = gray!40, opacity = 0.3] (#1+0.8*#3,#2) arc (0:180:0.4*#3 and -0.2*#3) -- (#1,#2-0.1*#3) arc (180:360:0.4*#3 and 0.4*#3) -- (#1+0.8*#3,#2);
\draw (#1+0.8*#3,#2) arc (0:180:0.4*#3 and -0.2*#3) -- (#1,#2-0.1*#3) arc (180:360:0.4*#3 and 0.4*#3) -- (#1+0.8*#3,#2);
}

\newcommand{\bowlud}[4]{
\shade[ball color = gray!40, opacity = 0.3] (#1,#2) arc (180:360:0.4*#3 and 0.2*#3) -- (#1+0.8*#3,#2+#4*#3) arc (0:180:0.4*#3 and 0.4*#3) -- (#1,#2);
\draw (#1,#2) arc (180:360:0.4*#3 and 0.2*#3) -- (#1+0.8*#3,#2+#4*#3) arc (0:180:0.4*#3 and 0.4*#3) -- (#1,#2);
\draw[dashed] (#1+0.8*#3,#2) arc (0:180:0.4*#3 and 0.2*#3);
}

\newcommand{\cylinder}[4]{
\shade[ball color = gray!40, opacity = 0.3] (#1,#2) arc (180:360:0.4*#3 and 0.2*#3) -- (#1+0.8*#3,#2+#4*#3) 
arc (0:180:0.4*#3 and -0.2*#3);
\shade[ball color = gray!40, opacity = 0.15] (#1,#2+#4*#3) arc (180:360:0.4*#3 and -0.2*#3) arc (0:180:0.4*#3 and -0.2*#3);
\draw (#1,#2) arc (180:360:0.4*#3 and 0.2*#3) -- (#1+0.8*#3,#2+#4*#3) 
arc (0:180:0.4*#3 and 0.2*#3) -- (#1,#2);
\draw[dashed] (#1+0.8*#3,#2) arc (0:180:0.4*#3 and 0.2*#3);
\draw (#1,#2+#4*#3) arc (180:360:0.4*#3 and 0.2*#3);
}

\newcommand{\cylindert}[4]{
\shade[ball color = gray!40, opacity = 0.3] (#1,#2) arc (180:360:0.4*#3 and 0.2*#3) -- (#1+0.8*#3,#2+#4*#3) 
arc (0:180:0.4*#3 and -0.2*#3);
\draw (#1, #2+#4 * #3) -- (#1,#2) arc (180:360:0.4*#3 and 0.2*#3) -- (#1+0.8*#3,#2+#4*#3) ;
\draw[dashed] (#1+0.8*#3,#2) arc (0:180:0.4*#3 and 0.2*#3);
}

\newcommand{\pop}[3]{
\shade[ball color = gray!40, opacity = 0.3] (#1,#2) arc (180:360:0.4*#3 and 0.2*#3) to [out = 90, in = 270] (#1+1.6*#3,#2+1*#3) arc (0:180:0.4*#3 and -0.2*#3) arc (0:180:0.4*#3 and -0.4*#3) arc (0:180:0.4*#3 and -0.2*#3) to [out = 270, in = 90] (#1,#2);
\shade[ball color = gray!40, opacity = 0.15] (#1-0.8*#3,#2+#3) arc (180:360:0.4*#3 and -0.2*#3) arc (0:180:0.4*#3 and -0.2*#3);
\shade[ball color = gray!40, opacity = 0.15] (#1+0.8*#3,#2+#3) arc (180:360:0.4*#3 and -0.2*#3) arc (0:180:0.4*#3 and -0.2*#3);
\draw (#1,#2) arc (180:360:0.4*#3 and 0.2*#3) to [out = 90, in = 270] (#1+1.6*#3,#2+1*#3) arc (0:180:0.4*#3 and -0.2*#3) arc (0:180:0.4*#3 and -0.4*#3) arc (0:180:0.4*#3 and -0.2*#3) to [out = 270, in = 90] (#1,#2);
\draw[dashed] (#1,#2) arc (180:360:0.4*#3 and -0.2*#3);
\draw (#1,#2+#3) arc (0:180:0.4*#3 and 0.2*#3);
\draw (#1+1.6*#3,#2+#3) arc (180:360:-0.4*#3 and -0.2*#3);
}

\newcommand{\popd}[3]{
\shade[ball color = gray!40, opacity = 0.3] (#1,#2) arc (180:360:0.4*#3 and 0.2*#3) to [out = 90, in = 270] (#1+1.6*#3,#2+1*#3) arc (0:180:0.4*#3 and -0.2*#3) arc (0:180:0.4*#3 and -0.4*#3) arc (0:180:0.4*#3 and -0.2*#3) to [out = 270, in = 90] (#1,#2);
\draw (#1,#2) arc (180:360:0.4*#3 and 0.2*#3) to [out = 90, in = 270] (#1+1.6*#3,#2+1*#3) arc (0:180:0.4*#3 and -0.2*#3) arc (0:180:0.4*#3 and -0.4*#3) arc (0:180:0.4*#3 and -0.2*#3) to [out = 270, in = 90] (#1,#2);
\draw[dashed] (#1,#2) arc (180:360:0.4*#3 and -0.2*#3);
}

\newcommand{\popud}[3]{
\shade[ball color = gray!40, opacity = 0.3] (#1,#2) arc (180:360:0.4*#3 and 0.2*#3) to [out = 270, in = 90] (#1+1.6*#3,#2-1*#3) arc (0:180:0.4*#3 and -0.2*#3) arc (0:180:0.4*#3 and 0.4*#3) arc (0:180:0.4*#3 and -0.2*#3) to [out = 90, in = 270] (#1,#2);
\shade[ball color = gray!40, opacity = 0.15] (#1,#2) arc (180:360:0.4*#3 and -0.2*#3) arc (0:180:0.4*#3 and -0.2*#3);
\draw (#1,#2) arc (180:360:0.4*#3 and 0.2*#3) to [out = 270, in = 90] (#1+1.6*#3,#2-1*#3) arc (0:180:0.4*#3 and -0.2*#3) arc (0:180:0.4*#3 and 0.4*#3) arc (0:180:0.4*#3 and -0.2*#3) to [out = 90, in = 270] (#1,#2);
\draw (#1,#2) arc (180:360:0.4*#3 and -0.2*#3);
\draw[dashed] (#1,#2-#3) arc (0:180:0.4*#3 and 0.2*#3);
\draw[dashed] (#1+1.6*#3,#2-#3) arc (180:360:-0.4*#3 and -0.2*#3);
}

\newcommand{\popudd}[3]{
\shade[ball color = gray!40, opacity = 0.3] (#1,#2) arc (180:360:0.4*#3 and 0.2*#3) to [out = 270, in = 90] (#1+1.6*#3,#2-1*#3) arc (0:180:0.4*#3 and -0.2*#3) arc (0:180:0.4*#3 and 0.4*#3) arc (0:180:0.4*#3 and -0.2*#3) to [out = 90, in = 270] (#1,#2);
\draw (#1,#2) arc (180:360:0.4*#3 and 0.2*#3) to [out = 270, in = 90] (#1+1.6*#3,#2-1*#3) arc (0:180:0.4*#3 and -0.2*#3) arc (0:180:0.4*#3 and 0.4*#3) arc (0:180:0.4*#3 and -0.2*#3) to [out = 90, in = 270] (#1,#2);
\draw[dashed] (#1,#2-#3) arc (0:180:0.4*#3 and 0.2*#3);
\draw[dashed] (#1+1.6*#3,#2-#3) arc (180:360:-0.4*#3 and -0.2*#3);
}

\newcommand{\cancelot}[3]{
\draw (#1,#2) arc (180:360:0.4*#3 and 0.2*#3) to [out = 90, in = 270] (#1+1.6*#3,#2+1*#3) arc (0:180:0.4*#3 and -0.2*#3) arc (0:180:0.4*#3 and -0.4*#3) arc (0:180:0.4*#3 and -0.2*#3) to [out = 270, in = 90] (#1,#2);
\draw[dashed] (#1,#2) arc (180:360:0.4*#3 and -0.2*#3);
\draw[dashed] (#1,#2+#3) arc (0:180:0.4*#3 and 0.2*#3);
\draw[dashed] (#1+1.6*#3,#2+#3) arc (180:360:-0.4*#3 and -0.2*#3);
\draw (#1+0.8*#3,#2+#3) -- (#1+0.8*#3,#2+1.1*#3) arc (180:360:0.4*#3 and -0.4*#3) -- (#1+1.6*#3,#2+#3);
\draw (#1, #2+#3) to [out = 90, in = 270] (#1+0.8*#3, #2+2*#3) arc (0:180:0.4*#3 and -0.2*#3) to [out = 270, in = 90] (#1-0.8*#3,#2+#3);
\draw (#1,#2+2*#3) arc (180:360:0.4*#3 and -0.2*#3);
\shade[ball color = gray!40, opacity = 0.3] (#1+0.8*#3,#2+#3) arc (180:360:0.4*#3 and 0.2*#3) -- (#1+1.6*#3,#2+1.1*#3) arc (0:180:0.4*#3 and 0.4*#3) -- (#1+0.8*#3,#2+#3);
\shade[ball color = gray!40, opacity = 0.15] (#1,#2+2*#3) arc (180:360:0.4*#3 and -0.2*#3) arc (0:180:0.4*#3 and -0.2*#3);
\shade[ball color = gray!40, opacity = 0.3] (#1,#2) arc (180:360:0.4*#3 and 0.2*#3) to [out = 90, in = 270] (#1+1.6*#3,#2+1*#3) arc (0:180:0.4*#3 and -0.2*#3) arc (0:180:0.4*#3 and -0.4*#3) to [out = 90, in = 270] (#1+0.8*#3, #2+2*#3) arc (0:180:0.4*#3 and -0.2*#3) to [out = 270, in = 90] (#1-0.8*#3,#2+#3) to [out = 270, in = 90] (#1,#2);
}

\newcommand{\cancelzo}[3]{
\shade[ball color = gray!40, opacity = 0.15] (#1,#2) arc (180:360:0.4*#3 and -0.2*#3) arc (0:180:0.4*#3 and -0.2*#3);
\draw (#1,#2) arc (180:360:0.4*#3 and 0.2*#3) to [out = 270, in = 90] (#1+1.6*#3,#2-1*#3) arc (0:180:0.4*#3 and -0.2*#3) arc (0:180:0.4*#3 and 0.4*#3) arc (0:180:0.4*#3 and -0.2*#3) to [out = 90, in = 270] (#1,#2);
\draw (#1,#2) arc (180:360:0.4*#3 and -0.2*#3);
\draw[dashed] (#1,#2-#3) arc (0:180:0.4*#3 and 0.2*#3);
\draw[dashed] (#1+1.6*#3,#2-#3) arc (180:360:-0.4*#3 and -0.2*#3);
\draw (#1, #2-#3) -- (#1, #2-1.1*#3) arc (180:360:-0.4*#3 and 0.4*#3) -- (#1-0.8*#3, #2-#3);
\draw (#1+0.8*#3, #2-#3) to [out = 270, in = 90] (#1,#2-2*#3) arc (0:180:-0.4*#3 and -0.2*#3) to [out = 90, in = 270] (#1+1.6*#3, #2-#3);
\draw[dashed] (#1,#2-2*#3) arc (180:360:0.4*#3 and -0.2*#3);
\shade[ball color = gray!40, opacity = 0.3] (#1,#2-#3) arc (0:180:0.4*#3 and -0.2*#3) -- (#1-0.8*#3,#2-1.1*#3) arc (180:360:0.4*#3 and 0.4*#3) -- (#1,#2-#3);
\shade[ball color = gray!40, opacity = 0.3] (#1,#2) arc (180:360:0.4*#3 and 0.2*#3) to [out = 270, in = 90] (#1+1.6*#3,#2-1*#3) to [out = 270, in = 90] (#1+0.8*#3,#2-2*#3) arc (0:180:0.4*#3 and -0.2*#3) to [out = 90, in = 270] (#1+0.8*#3, #2-#3) arc (0:180:0.4*#3 and 0.4*#3) arc (0:180:0.4*#3 and -0.2*#3) to [out = 90, in = 270] (#1,#2);
}

\newcommand{\plane}[2]{
\shade[color = gray!40, opacity = 0.3] (#1,#2-0.5) -- (#1+1,#2) -- (#1+1,#2+1.5) -- (#1,#2+1) -- (#1,#2-0.5);
\draw (#1,#2-0.5) -- (#1+1,#2) -- (#1+1,#2+1.5) -- (#1,#2+1) -- (#1,#2-0.5);
}

\newcommand{\sphere}[3]{
\shade[ball color = gray!40, opacity = 0.3] ({#1},{#2}) circle ({#3});
\draw (#1,#2) circle ({#3});
\draw (#1-#3,#2) arc (180:360:#3 and 0.3*#3);
\draw[dashed] (#1+#3,#2) arc (0:180:#3 and 0.3*#3);
}

\newcommand{\spheredot}[3]{
\sphere{#1}{#2}{#3}
\node at (#1-0.3*#3,#2+0.45*#3) [scale = 1.5*#3] {$\bullet$};
}

\newcommand{\birth}[4]{
\draw[#4] (#1,#2) circle (0.55*#3);
\draw (#1- 0.4* #3,#2-0.4*#3) -- (#1-0.7*#3, #2-0.7*#3);
\draw (#1+ 0.4* #3,#2-0.4*#3) -- (#1+0.7*#3, #2-0.7*#3);
\draw (#1- 0.4* #3,#2+0.4*#3) -- (#1-0.7*#3, #2+0.7*#3);
\draw (#1+ 0.4* #3,#2+0.4*#3) -- (#1+0.7*#3, #2+0.7*#3);
}

\newcommand{\death}[4]{
\draw[#4] (#1,#2) circle (#3);
\draw (#1- 0.4* #3,#2-0.4*#3) -- (#1-0.7*#3, #2-0.7*#3);
\draw (#1+ 0.4* #3,#2-0.4*#3) -- (#1+0.7*#3, #2-0.7*#3);
\draw (#1- 0.4* #3,#2+0.4*#3) -- (#1-0.7*#3, #2+0.7*#3);
\draw (#1+ 0.4* #3,#2+0.4*#3) -- (#1+0.7*#3, #2+0.7*#3);
}

\newcommand{\relsplit}[4]{
\draw[#4] (#1,#2) -- (#1+0.1*#3,#2+0.1*#3) to [out = 45, in = 135] (#1+0.5*#3,#2+0.1*#3) to [out = 315, in = 225] (#1+0.8*#3,#2+0.1*#3) to [out = 45, in = 315] (#1+0.8*#3,#2+0.5*#3) to [out = 135, in = 45] (#1+0.5*#3, #2+0.5*#3) to [out = 225, in = 315] (#1+0.1*#3,#2+0.5*#3) -- (#1,#2+0.6*#3);
}

\newcommand{\relmerge}[4]{
\draw[#4] (#1,#2) to [out = 45, in = 315] (#1,#2+0.6*#3);
\draw[#4] (#1+0.6*#3,#2+0.3*#3) circle (0.25*#3);
}

\newcommand{\torusv}[5]{
\draw[#5] (#1,#2) -- (#1,#2+0.1*#4) to [out = 90, in = 270] (#1+0.1*#3, #2+0.4*#4) -- (#1+0.1*#3, #2+0.6*#4) to [out = 90, in = 270] ((#1,#2+0.9*#4) -- (#1,#2+#4);
\draw[#5] (#1+#3,#2) -- (#1+#3,#2+0.1*#4) to [out = 90, in = 270] (#1+0.9*#3, #2+0.4*#4) -- (#1+0.9*#3, #2+0.6*#4) to [out = 90, in = 270] (#1+#3,#2+0.9*#4) -- (#1+#3,#2+#4);
}

\newcommand{\torush}[5]{
\draw[#5] (#1,#2) -- (#1,#2+0.1*#4) to [out = 90, in = 180] (#1+0.4*#3,#2+0.3*#4) -- (#1+0.6*#3,#2+0.3*#4) to [out = 0, in = 90] (#1+#3, #2+0.1*#4) -- (#1+#3, #2);
\draw[#5] (#1,#2+#4) -- (#1,#2+0.9*#4) to [out = 270, in = 180] (#1+0.4*#3,#2+0.7*#4) -- (#1+0.6*#3,#2+0.7*#4) to [out = 0, in = 270] (#1+#3,#2+0.9*#4) -- (#1+#3,#2+#4);
}

\newcommand{\smalpha}[5] {
\draw[#5] (#1, #2) -- (#1, #2+ 0.2 * #4) to [out = 90, in = 270] (#1+#3, #2+0.8 * #4) -- (#1+#3, #2+#4);
\draw[#5] (#1, #2+#4) -- (#1, #2+ 0.8 * #4) to [out = 270, in = 270] (#1+0.5*#3, #2 + 0.8 *#4) -- (#1+0.5*#3, #2+#4);
\draw[#5] (#1+0.5 * #3, #2) -- (#1+0.5 * #3, #2 + 0.2 *#4) to [out = 90, in = 90] (#1+#3, #2 + 0.2 *#4) -- (#1+#3, #2);
}

\newcommand{\smbeta}[5] {
\draw[#5] (#1, #2+#4) -- (#1, #2+0.8*#4) to [out = 270, in = 90] (#1+#3, #2+0.2 * #4) -- (#1+#3, #2);
\draw[#5] (#1,#2) -- (#1, #2 + 0.2 *#4) to [out = 90, in = 90] (#1+0.5*#3, #2 + 0.2 *#4) -- (#1+0.5*#3, #2);
\draw[#5] (#1+0.5 * #3, #2+#4) -- (#1+0.5 * #3, #2+ 0.8 * #4) to [out = 270, in = 270] (#1+#3, #2 + 0.8 *#4) -- (#1+#3, #2+#4);
}

\newcommand{\smgamma}[5]{
\draw[#5] (#1,#2) -- (#1, #2 + 0.2 *#4) to [out = 90, in = 90] (#1+0.5*#3, #2 + 0.2 *#4) -- (#1+0.5*#3, #2);
\draw[#5] (#1, #2+#4) -- (#1, #2+ 0.8 * #4) to [out = 270, in = 270] (#1+0.5*#3, #2 + 0.8 *#4) -- (#1+0.5*#3, #2+#4);
\draw[#5] (#1+#3, #2) -- (#1+#3, #2+#4);
}

\newcommand{\smdelta}[5]{
\draw[#5] (#1,#2) -- (#1, #2+#4);
\draw[#5] (#1+0.5 * #3, #2) -- (#1+0.5 * #3, #2 + 0.2 *#4) to [out = 90, in = 90] (#1+#3, #2 + 0.2 *#4) -- (#1+#3, #2);
\draw[#5] (#1+0.5 * #3, #2+#4) -- (#1+0.5 * #3, #2+ 0.8 * #4) to [out = 270, in = 270] (#1+#3, #2 + 0.8 *#4) -- (#1+#3, #2+#4);
}

\newcommand{\smomega}[5]{
\draw[#5] (#1,#2) -- (#1, #2+#4);
\draw[#5] (#1+0.5 * #3, #2) -- (#1+0.5 * #3, #2+#4);
\draw[#5] (#1+#3, #2) -- (#1+#3, #2+#4);
}

\newcommand{\closurex}[5]{
\draw[#5] (#1,#2) -- (#1, #2-0.1*#4) to [out = 270, in = 270] (#1-0.2*#3, #2-0.1*#4) -- (#1-0.2*#3, #2+1.1*#4) to [out = 90, in = 90] (#1, #2+1.1 * #4) -- (#1, #2+#4);
\draw[#5] (#1+#3,#2) -- (#1+#3, #2-0.1*#4) to [out = 270, in = 270] (#1+1.2*#3, #2-0.1*#4) -- (#1+1.2*#3, #2+1.1*#4) to [out = 90, in = 90] (#1+#3, #2+1.1 * #4) -- (#1+#3, #2+#4);
\draw[#5] (#1+0.5 * #3, #2) -- (#1+0.5 * #3, #2-0.1*#4) to [out = 270, in = 270] (#1+1.4*#3, #2-0.1*#4) -- (#1+1.4*#3, #2+1.1*#4) to [out = 90, in = 90] (#1+0.5 * #3, #2+1.1 * #4) -- (#1+0.5 * #3, #2+#4);
}

\newcommand{\dottedcup}[5]{
\draw[#5] (#1, #2+#4) -- (#1, #2+ 0.8 * #4) to [out = 270, in = 270] (#1+0.5*#3, #2 + 0.8 *#4) -- (#1+0.5*#3, #2+#4);
\node[scale = 1.4 * #3] at (#1+0.27 * #3, #2 + 0.675 *#4) {$\bullet$};
}

\newcommand{\dottedcap}[5] {
\draw[#5] (#1,#2) -- (#1, #2 + 0.2 *#4) to [out = 90, in = 90] (#1+0.5*#3, #2 + 0.2 *#4) -- (#1+0.5*#3, #2);
\node[scale = 1.4 * #3] at (#1+0.27 * #3, #2+ 0.325 * #4) {$\bullet$};
}

\newcommand{\dottedbslash}[5] {
\draw[#5] (#1, #2+#4) -- (#1, #2+0.8*#4) to [out = 270, in = 90] (#1+#3, #2+0.2 * #4) -- (#1+#3, #2);
\node[scale = 1.4 * #4] at (#1 + 0.5 * #3, #2+0.5 * #4) {$\bullet$};
}

\newcommand{\dottedslash}[5] {
\draw[#5] (#1, #2) -- (#1, #2+ 0.2 * #4) to [out = 90, in = 270] (#1+#3, #2+0.8 * #4) -- (#1+#3, #2+#4);
\node[scale = 1.4 * #4] at (#1 + 0.5 * #3, #2+0.5 * #4) {$\bullet$};
}

\newcommand{\dottedline}[4] {
\draw[#4] (#1, #2) -- (#1, #2+#3);
\node[scale = 1.4 * #3] at (#1+0.02 * #3,  #2+0.5 * #3) {$\bullet$};
}

\begin{document}
\parindent0em
\setlength\parskip{.1cm}
\thispagestyle{empty}
\title{A scanning Algorithm for odd Khovanov homology}
\author[Dirk Sch\"utz]{Dirk Sch\"utz}
\address{Department of Mathematical Sciences\\ Durham University\\ United Kingdom}
\email{dirk.schuetz@durham.ac.uk}

\begin {abstract}
We adapt Bar-Natan's scanning algorithm for fast computations in (even) Khovanov homology to odd Khovanov homology. We use a mapping cone construction instead of a tensor product, which allows us to deal efficiently with the more complicated sign assignments in the odd theory. The algorithm has been implemented in a computer program. We also use the algorithm to determine the odd Khovanov homology of $3$-strand torus links. 
\end {abstract}

\maketitle

\section{Introduction}
In the highly influental paper \cite{MR1740682}, Khovanov introduced his knot homology as a categorification of the Jones polynomial. A few years later, Ozsv\'{a}th, Rasmussen, and Szab\'{o} \cite{MR3071132} gave a different categorification, called odd Khovanov homology, using exterior algebras in place of symmetric algebras.  This homology agrees with Khovanov homology when coefficients in $\Z/2\Z$ are considered, but not over $\Q$. We refer to the original Khovanov homology here as even Khovanov homology.

While the construction in \cite{MR3071132} is similar to Khovanov's original construction, there are differences which make the odd theory somewhat more difficult to work with. Particularly noteworthy here are the sign assignments, which govern the chain complex condition $\partial^2=0$ in both even and odd theories. Unlike in the even case, for odd Khovanov homology the sign assignments depend on the link diagram as well as another orientation choice not needed in even Khovanov homology.

A very useful generalization of even Khovanov homology has been the extension to tangles, with different constructions given by Khovanov \cite{MR1928174} and Bar-Natan \cite{MR2174270}. An important feature in these papers is the behaviour under gluings of tangles, which can essentially be described by tensor products. Trying to extend odd Khovanov homology to tangles similarly leads to difficulties with the sign assignments. Still, Naisse and Putyra \cite{naisse2020odd} recently gave a construction for tangles with good gluing behaviour, based on earlier work of Putyra \cite{MR3363817}. For another construction, see also Vaz \cite{vaz2019khovanov}.

Bar-Natan used his tangle invariant to produce a fast algorithm for calculating even Khovanov homology \cite{MR2320156}. Our goal in this paper is to extend this algorithm to odd Khovanov homology. Bar-Natan's algorithm crucially makes use of the gluing behaviour via tensor product, so we need a suitable substitute for this. 

Rather than trying to use a tangle invariant combined with a tensor product, we iterate a mapping cone construction whilst scanning through the crossings of a given link diagram. We note that the tensor product in Bar-Natan's scanning algorithm, which involves one chain complex to be concentrated in two adjacent homological degrees, can be interpreted as a mapping cone. But by taking this slightly different point of view on Bar-Natan's scanning algorithm, we get the flexibility needed to deal with the more difficult sign assignments in odd Khovanov homology.

Our approach is less sophisticated than the tangle constructions mentioned above.  In fact, the intermediate chain complexes we obtain do depend on the original link diagram. Nevertheless, the simplicity of our approach means that odd Khovanov homology can now be calculated very similarly to even Khovanov homology. An implementation of the algorithm is available through the author's webpage. 
As a sample computation, we list the odd Khovanov homology of the $(8,9)$-torus knot in Figures \ref{fig:torus892} and \ref{fig:torus893}.

Rational computations can already be found in \cite{MR3071132}, and integral computations were done by Shumakovitch \cite{MR2777025} who observed that torsion of order different from 2 is far more common in odd than in even Khovanov homology. Already the $(3,4)$-torus knot contains torsion of order $3$ in its odd Khovanov homology. We use our algorithm to show that this holds for every $(3,n)$-torus link with $n\geq 4$.

Finally, we use our algorithm to calculate a concordance invariant based on the first Steenrod square arising from the work of \cite{MR4078823}.

\section{A recap of odd Khovanov homology}\label{sec:recap}
Let us recall how odd Khovanov homology is defined. Our presentation follows \cite{MR3071132} closely and we refer the reader there for more information. Let $\mathcal{D}$ be an oriented link diagram, where every crossing has an arrow going through it as in the left of Figure \ref{fig:crossing_arrow}. We refer to the left smoothing in Figure \ref{fig:crossing_arrow} as the $0$-smoothing, and the right one as the $1$-smoothing. So if the diagram $\mathcal{D}$  has $n$ crossings, we get $2^n$ smoothings, and by choosing an order of these crossings, we can parametrize these smoothings by elements $c=(c_1,\ldots,c_n)\in \{0,1\}^n$ and write $S_c$ for the smoothing, a disjoint union of circles.

\begin{figure}[ht]
\begin{tikzpicture}
\crossing{0}{0}{1}{very thick}
\arroweast{0.3}{0.5}{0.4}
\node at (1.5,0.5) {:};
\smoothingup{2}{0}{1}{very thick}
\arroweast{3.3}{0.4}{1.4}
\smoothingup{3.7}{0.6}{0.6}{thick}
\arroweast{3.9}{0.9}{0.2}
\smoothinglr{5}{0}{1}{very thick}
\end{tikzpicture}
\caption{\label{fig:crossing_arrow}A crossing with associated smoothings.}
\end{figure}
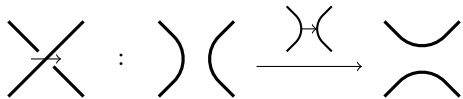 

We refer to the $c\in \{0,1\}^n$ as the vertices in the {\em hypercube of oriented resolutions}. We think of the (large) arrow between the two smoothings in Figure \ref{fig:crossing_arrow} as a surgery, and the small arrow as the surgery arc. The orientation of this arc will become important shortly. Note that the vertices $c,c'$ of the two smoothings involved in a surgery only differ in one coordinate $i$, and we can parametrize the surgeries by elements $e^i_c=(c_1,\ldots,c_{i-1},\ast,c_i,\ldots,c_{n-1})$ with $c=(c_1,\ldots,c_{n-1})\in \{0,1\}^{n-1}$.
Here $\ast$ is simply a symbol different from $0$ and $1$, and we think of such an $e_c$ as an oriented edge in the hypercube. The orientation points from the vertex $c^0$, obtained from $e^i_c$ by changing $\ast$ to $0$, to the vertex $c^1$ obtained by changing $\ast$ to $1$.

Given $c$ a vertex in the hypercube, we write 
\[
|c| = \sum_{i=1}^n c_i \in \Z,
\]
and let $\Lambda^\ast(S_c)$ be the exterior algebra over the free abelian group generated by the components of $S_c$. If $e_c$ is an edge in the hypercube between two vertices $c^0$ and $c^1$ such that $|c^1| = |c^0|+1$, then $S_{c^1}$ is obtained from $S_{c^0}$ by a surgery which either merges two components, or splits one component into two.

We then get a homomorphism
\[
F_{e_c}\colon \Lambda^\ast(S_{c^0}) \to \Lambda^\ast(S_{c^1})
\]
which in the case of a merger can be identified with the quotient map induced by identifying the two merging components. In the case of a split, the map is given by left-multiplication with $(s_1-s_0)$, where $s_1$ and $s_0$ are the two components arising in the split. To distinguish $s_1$ from $s_0$ we use the small arrow in Figure \ref{fig:crossing_arrow} as follows. If we rotate this arrow by 90 degrees anticlockwise, the arrow points from $s_0$ to $s_1$. Specifically in Figure \ref{fig:crossing_arrow} on the right smoothing the lower component is $s_0$, while the upper component is $s_1$.

We can now construct cochain groups $C^i(\mathcal{D})$ by
\[
C^i(\mathcal{D}) = \bigoplus_{c, |c| = i} \Lambda^\ast(S_c)
\]
and we want to use the $F_{e_c}$ between the various direct summands to get the coboundary. In order to get $\delta^2=0$ we need to take a closer look at faces in the hypercube.

Assume that $c\in \{0,1\}^{n-2}$ and $1\leq i < j \leq n$. We let
\[
f^{ij}_c = (c_1,\ldots,c_{i-1},\ast,c_i,\ldots, c_{j-2},\ast,c_{j-1},\ldots, c_{n-2}).
\]
We call this a face in the hypercube between the four edges $e^{0\ast}, e^{1\ast}, e^{\ast 0}, e^{\ast 1}$ obtained by replacing one of the $\ast$ with either $0$ or $1$, and with vertices $c^{00},c^{01}, c^{10}, c^{11}$. Given such a face, we get from $S_{c^{00}}$ to $S_{c^{11}}$ using two surgeries, and depending on the order in which we perform these surgeries, we get two homomorphisms
\[
F_{e^{1\ast}}\circ F_{e^{\ast0}}, F_{e^{\ast1}}\circ F_{e^{0\ast}}\colon \Lambda^\ast(S_{c^{00}})\to \Lambda^\ast(S_{c^{11}}).
\]
Depending on the local surgery picture, these two homomorphisms are equal, differ by a factor $-1$, or are both $0$. This leads to four types for these surgery pictures, and they are listed in Figure \ref{fig:subcases}.

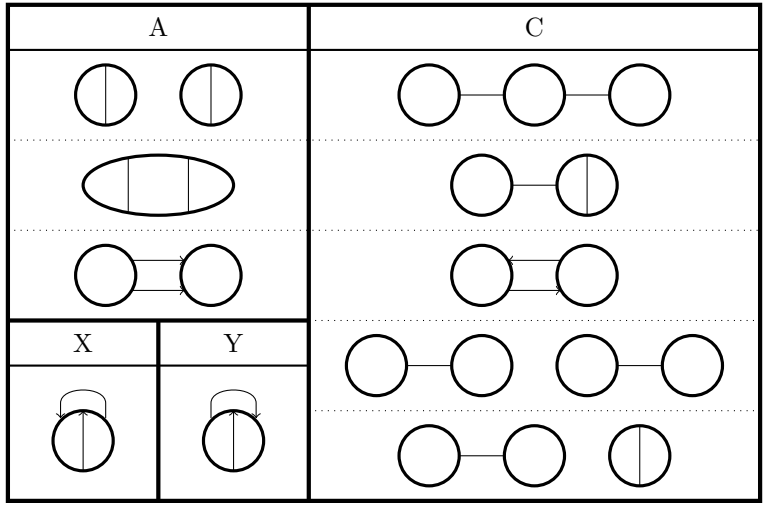
\begin{figure}[ht]
\begin{tikzpicture}
\draw[ultra thick] (0,0.4) rectangle (10,7);
\draw[thick] (0,6.4) -- (10,6.4);
\draw[ultra thick] (4,0.4) -- (4,7);
\node at (2,6.7) {A};
\node at (1, 2.5) {X};
\node at (3, 2.5) {Y};
\node at (7, 6.7) {C};
\draw[very thick] (1.3,5.8) circle (0.4);
\draw (1.3, 5.4) -- (1.3, 6.2);
\draw[very thick] (2.7,5.8) circle (0.4);
\draw (2.7, 5.4) -- (2.7, 6.2);
\draw[very thick] (2, 4.6) ellipse (1 and 0.4);
\draw (1.6, 4.25) -- (1.6, 4.95);
\draw (2.4, 4.25) -- (2.4, 4.95);
\draw[very thick] (1.3,3.4) circle (0.4);
\draw[very thick] (2.7,3.4) circle (0.4);
\draw[->] (1.65, 3.2) -- (2.35, 3.2);
\draw[->] (1.65, 3.6) -- (2.35, 3.6);
\draw[very thick] (6.3,3.4) circle (0.4);
\draw[->] (6.65, 3.2) -- (7.35, 3.2);
\draw[<-] (6.65, 3.6) -- (7.35, 3.6);
\draw[very thick] (7.7,3.4) circle (0.4);
\draw[very thick] (6.3,4.6) circle (0.4);
\draw (6.7, 4.6) -- (7.3, 4.6);
\draw (7.7, 4.2) -- (7.7, 5);
\draw[very thick] (7.7,4.6) circle (0.4);
\draw[very thick] (7, 5.8) circle (0.4);
\draw (6,5.8) -- (6.6, 5.8);
\draw[very thick] (5.6, 5.8) circle (0.4);
\draw (7.4, 5.8) -- (8, 5.8);
\draw[very thick] (8.4, 5.8) circle (0.4);
\draw[very thick] (6.3,2.2) circle (0.4);
\draw (5.3, 2.2) -- (5.9, 2.2);
\draw[very thick] (7.7,2.2) circle (0.4);
\draw[very thick] (4.9, 2.2) circle (0.4);
\draw (8.1, 2.2) -- (8.7, 2.2);
\draw[very thick] (9.1, 2.2) circle (0.4);
\draw[very thick] (7, 1) circle (0.4);
\draw (6,1) -- (6.6, 1);
\draw[very thick] (5.6, 1) circle (0.4);
\draw[very thick] (8.4, 1) circle (0.4);
\draw (8.4, 0.6) -- (8.4, 1.4);
\draw[dotted] (0,5.2) -- (10, 5.2);
\draw[dotted] (0, 4) -- (10,4);
\draw[dotted] (4, 2.8) -- (10, 2.8);
\draw[dotted] (4, 1.6) -- (10, 1.6);
\draw[ultra thick] (0, 2.8) -- (4, 2.8);
\draw[ultra thick] (2, 2.8) -- (2, 0.4);
\draw[thick] (0, 2.2) -- (4, 2.2);
\draw[very thick] (1,1.2) circle (0.4);
\draw[->] (1.3, 1.5) -- (1.3, 1.7) to [out = 90, in = 90] (0.7,1.7) -- (0.7, 1.5);
\draw[very thick] (3,1.2) circle (0.4);
\draw[<-] (3.3, 1.5) -- (3.3, 1.7) to [out = 90, in = 90] (2.7,1.7) -- (2.7, 1.5);
\draw[->] (1, 0.8) -- (1, 1.6);
\draw[->] (3, 0.8) -- (3, 1.6);
\end{tikzpicture}
\caption{\label{fig:subcases}Commutation chart: Thick lines represent components in $S_{c}$, thin lines the surgery arcs. If a surgery arc has no orientation, then both orientations lead to the same result. We refer to the types as A.1 to A.3, and C.1 to C.5 (read from top to bottom).}
\end{figure}

For the cases of type A we get the two homomorphisms to be non-zero, and differ by a factor $-1$, for the cases of type C, the two homomorphisms are equal and non-zero. In the remaining two types both homomorphisms are $0$.

We say that a face $f$ in the hypercube is of type A etc.\ whenever the corresponding surgery picture is of that type.

\begin{definition}
Let $\mathcal{D}$ be a link diagram, where every crossing has a chosen arrow. A {\em sign assignment} is a function $\varepsilon$ from the set of edges in the hypercube to $\Z/2\Z$, such that for every face $f$ in the hypercube we have
\[
\varepsilon(e^{0\ast})+\varepsilon(e^{1\ast})+\varepsilon(e^{\ast0})+\varepsilon(e^{\ast1}) = 0 \in \Z/2\Z
\]
whenever $f$ is of type A or X, and
\[
\varepsilon(e^{0\ast})+\varepsilon(e^{1\ast})+\varepsilon(e^{\ast0})+\varepsilon(e^{\ast1}) = 1 \in \Z/2\Z
\]
whenever $f$ is of type C or Y. Here $e^{0\ast}, e^{1\ast}, e^{\ast 0}, e^{\ast 1}$ are the four edges of the face.
\end{definition}

Sign assignments do always exist, see \cite[Lm.1.3]{MR3071132}, and we will see how to construct one in Section \ref{sec:signs}.

Given a sign assignment $\varepsilon$, we get a coboundary map $\delta\colon C^\ast(\mathcal{D})\to C^{\ast+1}(\mathcal{D})$ by using $(-1)^{\varepsilon(e_c)}F_{e_c}$ between the various direct summands. 

The cochain complex $C^\ast(\mathcal{D})$ is in fact bigraded. We get a second grading, called the $q$-grading, on each $\Lambda^\ast(S_c)$, by declaring that elements of $\Lambda^r(S_c)$ have $q$-degree $s - 2r$, where $s$ is the number of components in the $1$-manifold $S_c$.

To make the coboundary $q$-grading preserving, and to make the chain homotopy type independent of the link diagram, we need to shift the gradings appropriately. Let $n_+$ be the number of positive crossings in the oriented link diagram $\mathcal{D}$, and $n_-$ the number of negative crossings. We then define
\begin{equation}\label{eq:shifted}
CO^i(\mathcal{D}) = C^{i+n_-}(\mathcal{D})\{i-2n_-+n_+\},
\end{equation}
where $\{j\}$ indicates a shift in $q$-grading, so that an element of $C^\ast(\mathcal{D})$ with $q$-grading $k$ has $q$-grading $k+j$ when viewed as an element of $C^\ast(\mathcal{D})\{j\}$.

It is shown in \cite{MR3071132} that the cohomology of $CO^\ast(\mathcal{D})$ is independent of the various choices, and we denote the resulting link invariants by $\Kho{i,j}(L)$, with $i$ referring to the homological grading, and $j$ to the $q$-grading.

\section{Chronological cobordisms}

The construction of the previous section is essentially combinatorial, but as was already pointed out in \cite{MR3071132}, one can consider a functor from the category $\cC$ of closed, smooth $1$-dimensional manifolds with oriented cobordisms as morphisms, to the category of graded $\Z$-modules, by sending the $1$-manifold $S$ to $\Lambda^\ast(S)$, and sending cobordisms to compositions of the $F_e$ (together with appropriate homomorphisms for birth and death cobordisms).

However, this leads to sign ambiguities, and an alternative approach was suggested in \cite{MR3363817}. Particularly the dotted chronological cobordisms of \cite[\S 11]{MR3363817} are useful to us, and we present a simplified version here.


\begin{definition}
Let $S_0$ and $S_1$ be closed $1$-dimensional manifolds smoothly embedded into $\R^2$. A {\em chronological cobordism between $S_0$ and $S_1$} is a cobordism $W$ smoothly embedded into $\R^2\times [0,1]$ such that the following hold:
\begin{enumerate}
\item There exists an $\varepsilon>0$ such that $W\cap \R^2\times [0,\varepsilon] = S_0\times [0,\varepsilon]$ and\\ $W\cap \R^2\times [1-\varepsilon,1] = S_1\times [1-\varepsilon,1]$.
\item The height function $\tau\colon W\to [0,1]$ given by projection to the third coordinate is a Morse function such that $\tau^{-1}(\{c\})$ has exactly one point for each critical value $c$ of $\tau$.
\end{enumerate}
A {\em dotted chronological cobordism between $S_0$ and $S_1$} is a pair $(W,D)$, where $W$ is a chronological cobordism between $S_0$ and $S_1$, and $D\subset W \cap \R^2\times (0,1)$ is a finite set such that $\tau(x)\not=\tau(y)$ whenever $x,y\in C\cup D$ are different. Here $C\subset W$ is the set of critical points of $\tau$.
\end{definition}

If $(W,D)$ is a dotted chronological cobordism, there is a total order $<$ on $C\cup D$ induced by the values under $\tau$. We say that $x,y\in C\cup D$ are {\em adjacent}, if $x \not= y$ and there is no $z\in D\cup C -\{x,y\}$ between $x$ and $y$ in this order.

Given a dotted chronological cobordism between $S_0$ and $S_1$, and one between $S_1$ and $S_2$, we can get a dotted chronological cobordism between $S_0$ and $S_2$ by stacking them and rescaling. We want to have some basic identifications among such cobordisms, which among other things will make this associative. One thing we want to keep for now is the order on the critical points.

An  {\em order-preserving isotopy} between $W$ and $W'$ is a smooth map $\psi\colon W\times [0,1]\to \R^2\times [0,1]$ with $\psi(x,s) = x$ for $s$ near $0$ and $\psi(x,s)=\psi(x,1)$ for $s$ near $1$, such that each $\psi_t(W)$ for $t\in [0,1]$  is a dotted chronological cobordism with dotted set $D_t = \psi_t(D)$, and $(W',D') = (\psi_1(W),\psi_t(D))$. This allows us to move dots around on $W$, although in an order-preserving way.


\begin{definition}
A {\em framing} on a dotted chronological cobordism $(W,D)$ is a choice of orientation of a basis for each stable manifold $W^s(p)\subset W$, where $p$ is a critical point of $\tau$ of index $1$. A dotted chronological cobordism with a framing is called a {\em framed chronological cobordism}.
\end{definition}

A framing is determined by the choice of a tangent vector $X_p\in T_pW^s(p)\subset T_pW$. We will therefore visualize it with an arrow through the various critical points.

\begin{remark}
An order-preserving isotopy also isotopes a framing. We therefore put an equivalence relation on framed chronological cobordisms generated by order-preserving isotopies, and demanding that framings are respected.
\end{remark}

Let $(W,D)$ be a dotted chronological cobordism. We then define the {\em degree of $W$} as 
\[
\deg W = 2\cdot |D| - \sum_{p\in C} (-1)^{\ind p}.
\]

\begin{definition}
Let $\dChr(\emptyset)$ be the category whose objects are pairs $(S,q)$, where $S$ is a closed $1$-manifold embedded in $\R^2$ and $q\in \Z$, and where $\Mor((S_0,q_0),(S_1,q_1))$ are given by the equivalence classes of framed chronological cobordisms between $S_0$ and $S_1$ of degree $q_1-q_0$.
\end{definition}

The identity morphism is represented by the cylinder $S\times [0,1]$. Composition is given by stacking cobordisms on top of each other. We think of the integer $q$ as giving a $q$-degree to each object. 

This is of course much more rigid than if we would only consider dotted cobordisms as morphisms, but it makes it easy to construct functors from $\dChr(\emptyset)$ to other categories.

In particular, we get a functor $F$ from $\dChr(\emptyset)$ to the category of graded $\Z$-modules by $F(S,q) = \Lambda^\ast(S)\{+q\}$. To describe $F$ on morphisms, note that all dotted chronological cobordisms are compositions of the five basic cobordisms listed in Figure \ref{fig:basicCob}, together with cylinders that can permute the circles.


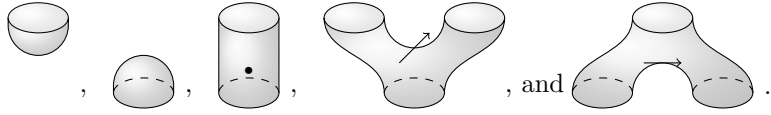
\begin{figure}[ht]
\begin{tikzpicture}
\bowl{0}{1}{1}{0.1}
\node at (1,0) {,};
\bowlud{1.4}{0}{1}{0.1}
\node at (2.4,0) {,};
\cylinder{2.8}{0}{1}{1}
\node at (3.2,0.3) [scale = 0.7] {$\bullet$};
\node at (3.8,0) {,};
\pop{5}{0}{1}
\popud{8.3}{1}{1}
\node at (7,0.1) {, and};
\node at (10.1,0) {.};
\draw[->] (5.2,0.4) -- (5.6,0.8);
\draw[->] (8.45, 0.4) -- (8.95,0.4);
\end{tikzpicture}
\caption{\label{fig:basicCob} The five basic cobordisms birth, death, dotting, split, and merge. Notice that for merge and split there also exist versions with the opposite framing.}
\end{figure}

Given the $1$-manifold $S$, we denote its components by $a,b,c,\cdots$ and use these letters also for the generators of the exterior algebra $\Lambda^\ast(S)$. A cylinder that permutes the components of $S$ is assigned the isomorphism of $\Lambda^\ast(S)$ that permutes the letters $a,b,c,\cdots$ in the way suggested by the cylinder.

The birth morphism from $(S,q)\to (S',q-1)$ is assigned the inclusion $\Lambda^\ast(S)\{q\}\to \Lambda^\ast(S')\{q-1\}$. 

For the death morphism $(S,q)\to (S',q-1)$ assume that the disappearing component is $a$. If $w = a_1\cdots a_k$ is a word in the components of $S$, it represents an element $\hat{w} = a_1\wedge \cdots \wedge a_k\in \Lambda^\ast(S)$. If $a=a_i$ for a unique $i$ we send $\hat{w}$ to $(-1)^{i-1} a_1\wedge \cdots \wedge a_{i-1}\wedge a_{i+1} \wedge \cdots \wedge a_k$. Otherwise it is send to $0$.

The dotting morphism $(S,q) \to (S,q+2)$ is send to left multiplication by $a$, where $a$ is the component on which the dot is.

For the split morphism we have a component $a$ which splits into two components $a_1$ and $a_2$. We would like to call one of these $a$, and use a new name $b$ for the other. To decide which one is $a$, take the tangent vector $X_p$ of the critical point which represents the framing, and rotate\footnote{Notice that the tangent space $T_pW =\R^2\times \{t\}$ for some $t\in (0,1)$, so `rotating to the left' is meant as a counter clockwise rotation in the plane.} it to the left until it is tangent to the unstable manifold $W^u(p)$. Now follow the unstable manifold in the direction of the rotated vector until we reach a component. This component is component $a$. The induced homomorphism on exterior algebras is now left multiplication by $a-b$. 

For the merge morphism two components $a,b$ of $S$ merge into one component $c$ of $S'$. The corresponding homomorphism on exterior algebras sends words $w$ to the word where each $a$ and $b$ is replaced by $c$. Notice that the framing has no influence in this case.

Let us turn $\dChr(\emptyset)$ into an additive category as in \cite[\S 3]{MR2174270}. That is, we first make it pre-additive by keeping the same objects, but turning morphism sets into free abelian groups generated by these sets. Then we consider formal direct summands of objects as the new objects, and morphisms as modelled on matrices. See \cite[Def.3.2]{MR2174270} for details. We call this new additive category $\Z\dChr(\emptyset)$.

The functor $F$ extends to $\Z\dChr(\emptyset)$, but we would like to have a functor from a category where the morphisms are not as rigid.

Let us introduce an equivalence relation $\sim$ on morphisms in $\Z\dChr(\emptyset)$ generated by the following. These pictures are local in the sense that they can be part of larger cobordisms. However, it is assumed that there are no other critical points or dottings on the shown levels.
\begin{equation}\label{eq:firstequiv}
\begin{tikzpicture}[baseline={([yshift=-.5ex]current bounding box.center)}, scale = 0.8]
\sphere{0}{0}{0.5}
\node at (1.1,0) {$\sim 0$,};
\spheredot{3}{0}{0.5}
\node at (4.1,0) {$\sim 1$,};
\plane{5}{-0.5}
\node at (5.5,0.2) [scale = 0.75] {$\bullet$};
\node at (5.5,-0.2) [scale = 0.75] {$\bullet$};
\node at (6.6,0) {$\sim 0$.};
\end{tikzpicture}
\end{equation}
Here $\sim 1$ means that any cobordism with a dotted sphere is equivalent to the one with this sphere removed.

We also want the neck-cutting relation already present in \cite{MR2174270}.
\begin{equation}\label{eq:neckcut}
\begin{tikzpicture}[baseline={([yshift=-.5ex]current bounding box.center)}, scale =1] 
\cylinder{0}{0}{1}{1.5}
\node at (1.25,0.75) {$\sim$};
\bowl{1.6}{1.5}{1}{0.1}
\bowlud{1.6}{0}{1}{0.1}
\node at (2,0.1) [scale = 0.75] {$\bullet$};
\node at (2.85,0.75) {$+$};
\bowl{3.2}{1.5}{1}{0.1}
\bowlud{3.2}{0}{1}{0.1}
\node at (3.6,1.2) [scale = 0.75] {$\bullet$};
\end{tikzpicture}
\end{equation}

Note that (\ref{eq:firstequiv}) and (\ref{eq:neckcut}) are exactly the relations used in \cite[\S 11.2]{MR2174270}. But right now our category is still very rigid because of the chronologies. We first want to loosen the behaviour of the framings.

\begin{equation}\label{eq:secondequiv}
\begin{tikzpicture}[baseline={([yshift=-.5ex]current bounding box.center)}]
\pop{0}{0}{0.8}
\pop{2.6}{0}{0.8}
\popud{6}{0.8}{0.8}
\popud{8.4}{0.8}{0.8}
\draw[->] (0.2,0.35) -- (0.5,0.65);
\draw[<-] (2.8,0.35) -- (3.1,0.65);
\node at (1.65,0.3) {$\sim -$};
\draw[->] (6.15,0.33) -- (6.55,0.33);
\draw[<-] (8.5,0.33) -- (8.9,0.33);
\node at (7.55,0.35) {$\sim$};
\node at (4, 0.1) {$,$};
\end{tikzpicture}
\end{equation}
Because of the right equivalence we can drop the arrow from pictures involving merges.

Consider the diagrams in Figure \ref{fig:subcases}. We can consider them as {\em surgery diagrams} which give rise to two chronological cobordisms $W$ and $\bar{W}$ who differ in the order of the index $1$ critical points. We then declare $W\sim\bar{W}$, if the corresponding surgery diagram is of type C, and $W\sim -\bar{W}$, if it is of type A. We do not need to consider types X and Y, as each cobordism will turn out to be equivalent to $0$.

Similarly, if we have two chronological cobordisms $W$ and $\bar{W}$ that both have two critical points, at least one of which has index $0$ or $2$, and the cobordisms agree without the chronology, we set $W\sim \bar{W}$, if there is a birth or a merge, and $W\sim -\bar{W}$ otherwise.

%
%
When canceling two adjacent critical points we add the following equivalences.
\begin{equation}\label{eq:cancelcrit}
\begin{tikzpicture}[baseline={([yshift=-.5ex]current bounding box.center)}, scale = 0.8]
\cancelot{0}{0}{1}
\cylinder{3}{0}{1}{2}
\cancelzo{6}{2}{1}
\draw[<-] (0.25,0.45) -- (0.55,0.75);
\node at (2.3,1) {$\sim$};
\node at (4.6,1) {$\sim$};
\end{tikzpicture}
\end{equation} 

To slide dots past critical points, we add the following equivalences.
\begin{equation}\label{eq:fifthequiv}
\begin{tikzpicture}[baseline={([yshift=-.5ex]current bounding box.center)}]
\bowl{0}{1}{0.75}{0.5}
\cylinder{0.8}{0}{0.75}{1.333}
\node at (1.125,0.7) [scale = 0.75] {$\bullet$};
\node at (1.75,0.5) {$\sim$};
\node at (3.6,0) {,};
\bowl{2.05}{1}{0.75}{0.1}
\cylinder{2.85}{0}{0.75}{1.333}
\node at (3.175,0.3) [scale = 0.75] {$\bullet$};
\cylinder{3.8}{0}{0.75}{1.333}
\node at (4.125,0.7) [scale = 0.75] {$\bullet$};
\cylinder{4.6}{0}{0.75}{1.333}
\node at (4.925,0.3) [scale = 0.75] {$\bullet$};
\node at (5.7,0.5) {$\sim -$};
\cylinder{6.1}{0}{0.75}{1.333}
\node at (6.425,0.3) [scale = 0.75] {$\bullet$};
\cylinder{6.9}{0}{0.75}{1.333}
\node at (7.225,0.7) [scale = 0.75] {$\bullet$};
\node at (7.65,0) {,};
\bowlud{7.85}{0}{0.75}{0.5}
\cylinder{8.65}{0}{0.75}{1.333}
\node at (8.975,0.3) [scale = 0.75] {$\bullet$};
\node at (9.75,0.5) {$\sim -$};
\bowlud{10.15}{0}{0.75}{0.1}
\cylinder{10.95}{0}{0.75}{1.333}
\node at (11.275,0.7) [scale = 0.75] {$\bullet$};
\end{tikzpicture}
\end{equation}
and
\begin{equation}\label{eq:sixthequiv}
\begin{tikzpicture}[baseline={([yshift=-.5ex]current bounding box.center)}]
\popud{0}{1}{1}
\node at (0.4,0.6) [scale = 0.75] {$\bullet$};
\node at (1.95,0.5) {$\sim$};
\popud{3}{1}{1}
\node at (2.6,0.1) [scale = 0.75] {$\bullet$};
\node at (4.75,0) {,};
\pop{6}{0}{1}
\node at (6.4,0.4) [scale = 0.75] {$\bullet$};
\node at (8.1,0.5) {$\sim -$};
\pop{9.3}{0}{1}
\node at (8.9,0.7) [scale = 0.75] {$\bullet$};
\end{tikzpicture}
\end{equation}

In the split cobordism the result does not depend on the framing.

\begin{lemma}
We have
\[
\begin{tikzpicture}
\cylinder{0}{0}{0.5}{2}
\node at (0.2,0.75) [scale = 0.6] {$\bullet$};
\popudd{1}{0.5}{0.5}
\cylinder{1}{0.5}{0.5}{1}
\node at (2.2,0.5) {$\sim$};
\cylinder{2.5}{0}{0.5}{2}
\node at (2.7,0.25) [scale = 0.6] {$\bullet$};
\popud{3.5}{1}{0.5}
\cylindert{3.1}{0}{0.5}{1}
\cylindert{3.9}{0}{0.5}{1}
\node at (4.8,0.5) {and};
\cylinder{5.3}{0}{0.5}{2}
\node at (5.5,0.75) [scale = 0.6] {$\bullet$};
\popd{6.3}{0}{0.5}
\cylinder{5.9}{0.5}{0.5}{1}
\cylinder{6.7}{0.5}{0.5}{1}
\node at (7.7,0.5) {$\sim \, -$};
\cylinder{8.2}{0}{0.5}{2}
\node at (8.4,0.25) [scale = 0.6] {$\bullet$};
\pop{9.2}{0.5}{0.5}
\cylindert{9.2}{0}{0.5}{1}
\node at (10.2,0.4) {.};
\end{tikzpicture}
\]
\end{lemma}

\begin{proof}
Both proofs are very similar. Introduce births and deaths using (\ref{eq:cancelcrit}), then use the neck-cutting relation to introduce a cylinder between the two components. Slide the dots using (\ref{eq:fifthequiv}) and (\ref{eq:sixthequiv}), and then remove the cylinder with the neck-cutting relation again.
\end{proof}

Taking the quotient by this equivalence relation gives us an additive category that we denote by $\Z\dChrq(\emptyset)$.

\begin{lemma}\label{lm:functor}
The functor $F$ induces a well defined functor from $\Z\dChrq(\emptyset)$ to the category of graded $\Z$-modules.
\end{lemma}

\begin{proof}
We need to check that all of the equivalences described lead to the same homomorphisms. Since there are a lot of cases, this is a bit lengthy, but none of the cases create any difficulties. As an example, let us check the first equivalence in (\ref{eq:cancelcrit}) preserves the homomorphism. Consider the cobordism where we first go through a split and then a death. Both domain and target are the same closed manifold $S$, and denote by $a$ the component that is visible. 
After the split we get a new component that we denote $b$, and because of the orientation of the stable manifold the homomorphism $\Lambda^\ast(S) \to \Lambda^\ast(S\cup b)$ is left multiplication by $(b-a)$. The homomorphism $\Lambda^\ast(S\cup b) \to \Lambda^\ast(S)$ induced by the death removes any $b$ from a word (with a sign depending on its position). Taking $w\in \Lambda^\ast(S)$ first leads to $b\wedge w - a\wedge w$, and removing $b$ from these words leads to $w$, since $a\wedge w$ does not contain any $b$. So the composition is the identity, as is the case for the cylinder.
\end{proof}

Given a link diagram $\mathcal{D}$ and a sign assignment $\varepsilon$ we can now form a cochain complex $[\mathcal{D}]$ over $\Z\dChrq(\emptyset)$, and applying the functor $F$ leads to the cochain complex from Section \ref{sec:recap}. 

Delooping \cite[Lm.4.1]{MR2320156} also works in this setting, compare \cite[Prop.11.4]{MR3363817}, so we should analyze $\Mor((\emptyset,q),(\emptyset,p))$ in $\Z\dChrq(\emptyset)$. We can move any dots close to the top, so any component with more than one dot leads to a zero morphism. Also, by the neck-cutting relation, (\ref{eq:sixthequiv}), and (\ref{eq:fifthequiv}), we have
\[
\begin{tikzpicture}
\popd{0}{-0.5}{0.5}
\popud{0}{1.5}{0.5}
\cylindert{-0.4}{0}{0.5}{2}
\cylindert{0.4}{0}{0.5}{2}
\node at (1.25,0.5) {$\sim$};
\popd{2}{-0.5}{0.5}
\popud{2}{1.5}{0.5}
\cylindert{1.6}{0}{0.5}{2}
\bowld{2.4}{1}{0.5}
\bowlud{2.4}{0}{0.5}{0.1}
\node at (4.6,0) [scale = 0.5] {$\bullet$};
\node at (3.25,0.5) {$+$};
\popd{4}{-0.5}{0.5}
\popud{4}{1.5}{0.5}
\cylindert{3.6}{0}{0.5}{2}
\bowld{4.4}{1}{0.5}
\bowlud{4.4}{0}{0.5}{0.1}
\node at (2.6,0.85) [scale = 0.5] {$\bullet$};
\node at (5.25,0.5) {$\sim$};
\popd{6}{-0.5}{0.5}
\popud{6}{1.5}{0.5}
\cylindert{5.6}{0}{0.5}{2}
\bowld{6.4}{1}{0.5}
\bowlud{6.4}{0}{0.5}{0.1}
\node at (8.225,-0.3) [scale = 0.5] {$\bullet$};
\node at (7.25,0.5) {$-$};
\popd{8}{-0.5}{0.5}
\popud{8}{1.5}{0.5}
\cylindert{7.6}{0}{0.5}{2}
\bowld{8.4}{1}{0.5}
\bowlud{8.4}{0}{0.5}{0.1}
\node at (6.225,1.3) [scale = 0.5] {$\bullet$};
\node at (9.4, 0.5) {$\sim 0$,};
\end{tikzpicture}
\]
so any surface of genus at least $1$ also leads to the zero morphism. So the morphism group from an emptyset to an emptyset is $\Z$ if $q=p$, and $0$ otherwise. 

Now let $B$ be a compact $2$-dimensional manifold smoothly embedded in $\R^2$, and let $\dot{B}\subset \partial B$ a finite subset such that $\dot{B}\cap C$ consists of an even number of points for every component $C$ of $\partial B$. 

For example, given a link diagram we could choose $B$ to be a disjoint union of small discs centered at (some of) the crossings of the diagram, and let $\dot{B}$ be the intersection of $\partial B$ with the diagram. 

Let $S$ be a compact $1$-manifold smoothly embedded in $B$ such that $\partial S = \dot{B}$ and $S\cap \partial B$ is transverse. We can define framed chronological cobordisms $(W,D)$ between such $S_0$ and $S_1$ as before, but we also require that $W\subset B\times [0,1]$ with $\partial W = W \cap (\partial B\times [0,1] \cup B\times \{0,1\})$ and that $W$ is a cylinder\footnote{This means that $S_0$ and $S_1$ have to agree near their common boundary for $W$ to exist.} near $\partial S_0\times [0,1]$.  So $W$ is a manifold with corners, but nothing interesting happens near the corner points. The definition of degree extends to such cobordisms, and we can also consider the same type of equivalence relation on them.

Define a category $\dChr(B,\dot{B})$ where the objects are pairs $(S,q)$ with $q\in \Z$ and $S$ a compact $1$-manifold smoothly embedded into $B$ such that $\partial S = \dot{B}$ and $S\cap \partial B$ is transverse, and morphisms between $(S_0,q_0)$ and $S_1,q_1)$ are given by framed chronological cobordisms between $S_0$ and $S_1$ of degree $q_1-q_0$.

Again we can turn this into an additive category $\Z\dChr(B,\dot{B})$. Of course we want to define an equivalence relation on the morphisms, and obtain a quotient category as before. All the equivalences described for $\Z\dChr(\emptyset)$ can be used as before, but we would also like to declare equivalences if some of the components are manifolds with corners. 

Before we do this let us introduce a {\em movie presentation} for chronological cobordisms. In Figure \ref{fig:basicmove} we list the moves for the basic cobordisms birth, death, dotting, and saddle. Any chronological cobordism can then be expressed as a finite sequence of these moves, which we write from left to right.

\begin{figure}[ht]
\begin{tikzpicture}
\birth{0}{0}{0.5}{very thick}
\death{1.5}{0}{0.5}{very thick}
\dotting{3}{0}{0.5}
\saddlenorth{4.5}{0}{0.5}
\node at (0.55,-0.45) {,};
\node at (2.05,-0.45) {,};
\node at (3.55,-0.45) {,};
\node at (5.05,-0.45) {.};
\end{tikzpicture}
\caption{\label{fig:basicmove} Movie presentations for birth, death, dotting, and saddle.}
\end{figure}
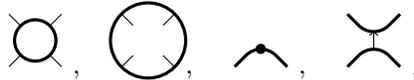


Since we require cobordisms to be cylinders near the corner points, we only consider births and deaths with the circle completely embedded in the interior of $B$. In the case of a saddle it may not be clear whether one has a merge or a split, and changing the orientation of the framing behaves differently for merges and splits. But note that merges and splits can be identified though, if there is only one interval ``missing''. So we can add the equivalences
\begin{equation}\label{eq:sevenequiv}
\begin{tikzpicture}[baseline={([yshift=-.5ex]current bounding box.center)}]
\relmerge{0}{0}{0.75}{thick}
\draw[->] (0.1,0.225) -- (0.275,0.225);
\node at (0.9,0.2) {$\sim$};
\relmerge{1.2}{0}{0.75}{thick}
\draw[<-] (1.3,0.225) -- (1.475,0.225);
\node at (2.4,0.225) {and};
\relsplit{3}{0}{0.75}{thick}
\draw[->] (3.225,0.3) -- (3.225,0.15);
\node at (4.1,0.2) {$\sim -$};
\relsplit{4.5}{0}{0.75}{thick}
\draw[<-] (4.725,0.3) -- (4.724,0.15);
\end{tikzpicture}
\end{equation}
Instead of (\ref{eq:cancelcrit}) we use the movie presentations
\begin{equation}\label{eq:eightequiv}
\begin{tikzpicture}[baseline={([yshift=-.5ex]current bounding box.center)}]
\relsplit{0}{0}{0.75}{thick}
\draw[->] (0.225,0.3) -- (0.225,0.15);
\draw[thick] (1,0) to [out = 45, in = 315] (1,0.45);
\death{1.4}{0.225}{0.2}{thick}
\draw[dashed] (-0.2,0.6) rectangle (1.8,-0.2);
\draw[dashed] (0.8,0.6) -- (0.8, -0.2);
\node at (2.4,0.225) {$\sim 1\sim$};
\draw[thick] (3.2,0) to [out = 45, in = 315] (3.2,0.45);
\birth{3.6}{0.225}{0.2}{thick}
\relmerge{4.2}{0}{0.75}{thick}
\draw (4.3,0.225) -- (4.475,0.225);
\draw[dashed] (3,0.6) rectangle (5,-0.2);
\draw[dashed] (4.0,0.6) -- (4.0,-0.2);
\end{tikzpicture}
\end{equation}
The relations of (\ref{eq:fifthequiv}) remain, but with the dotted cylinder over a circle replaced by dotted cylinders over an interval. Instead of (\ref{eq:sixthequiv}) we use
\begin{equation}\label{eq:ninthequiv}
\begin{tikzpicture}[scale = 0.77, baseline={([yshift=-.5ex]current bounding box.center)}]
\relmerge{0}{0}{0.75}{thick}
\draw(0.1,0.225) -- (0.275,0.225);
\draw[thick] (1,0) to [out = 45, in = 315] (1,0.45);
\node[scale = 0.7] at (1.1,0.225) {$\bullet$};
\draw[dashed] (-0.2,0.6) rectangle (1.3,-0.2);
\draw[dashed] (0.8,0.6) -- (0.8,-0.2);
\node at (1.65,0.225) {$\sim$}; 
\relmerge{-2.6}{0}{0.75}{thick}
\node[scale = 0.7] at (-2.5,0.225) {$\bullet$};
\relmerge{-1.6}{0}{0.75}{thick}
\draw (-1.5,0.225) -- (-1.325,0.225);
\node at (-0.5,0.225) {$\sim$};
\draw[dashed] (-2.8,0.6) rectangle (-0.8,-0.2);
\draw[dashed] (-1.8,0.6) -- (-1.8,-0.2);
\relmerge{2.15}{0}{0.75}{thick}
\node[scale = 0.7] at (2.425,0.225) {$\bullet$};
\relmerge{3.15}{0}{0.75}{thick}
\draw (3.25, 0.225) -- (3.425,0.225);
\draw[dashed] (1.95,0.6) rectangle (3.95,-0.2);
\draw[dashed] (2.95,0.6) -- (2.95, -0.2);
\node at (4.5, 0.225) {and};
\draw[thick] (8.2,0) to [out = 45, in = 315] (8.2,0.45);
\node[scale = 0.7] at (8.3,0.225) {$\bullet $};
\relsplit{8.7}{0}{0.75}{thick}
\draw[->] (8.925,0.3) -- (8.925,0.15);
\draw[dashed] (5,0.6) rectangle (7,-0.2);
\draw[dashed] (6,0.6) to (6,-0.2);
\node at (7.5,0.225) {$\sim -$};
\relsplit{5.2}{0}{0.75}{thick}
\draw[->] (5.425,0.3) -- (5.425,0.15);
\relmerge{6.2}{0}{0.75}{thick}
\node[scale = 0.7] at (6.3,0.225) {$\bullet $};
\draw[dashed] (8,0.6) rectangle (9.5,-0.2);
\draw[dashed] (8.5,0.6) -- (8.5,-0.2);
\node at (9.85,0.225) {$\sim$};
\relsplit{10.35}{0}{0.75}{thick}
\draw[->] (10.575,0.3) -- (10.575,0.15);
\relmerge{11.35}{0}{0.75}{thick}
\node[scale=0.7] at (11.625,0.225) {$\bullet$};
\draw[dashed] (10.15,0.6) rectangle (12.15,-0.2);
\draw[dashed] (11.15,0.6) -- (11.15,-0.2);
\end{tikzpicture}
\end{equation}
However, to get cochain complexes from the situations we want to consider, we need to add relations that indicate how a saddle with four boundary points commutes with other morphisms. We will make these relations depend on a closure of $B$. In order to do this assume that $X$ is a compact $1$-dimensional manifold smoothly embedded in $\R^2 - (B-\partial B)$ with $\partial X = \dot{B}$ and $X\cap \partial B$ is transverse.

We then get a functor $F_X\colon \Z\dChr(B,\dot{B})\to \Z\dChr(\emptyset)$ by gluing objects $S\subset B$ with $X$ along $\dot{B}$, and using product cobordisms $X\times [0,1]$ for morphisms. Note that for a given $S$ we slightly deform $X$ near its boundary to get a smooth embedding of a closed $1$-manifold.

Consider a surgery diagram in $B$ that involves two surgeries. As before, this gives rise to two chronological cobordisms $W$ and $\bar{W}$. We then declare $W\sim \bar{W}$ if the corresponding surgery diagram after closure with $X$ is of type C or Y, and $W\sim -\bar{W}$ if the corresponding closed surgery diagram is of type A or X.

Similarly, if $W$ is a saddle, and $\bar{W}$ a birth, death, or dotting, we get two chronological cobordisms $W_1$ and $W_2$ by combining them and allowing different orders. We then declare $W_1\sim (-1)^\delta W_2$ with $\delta\in\{0,1\}$ so that $F_X(W_1)\sim (-1)^\delta F_X(W_2)$.

Finally, if $W$ is a saddle with framing, and $\bar{W}$ the same saddle with opposite framing, we declare $W\sim (-1)^\delta \bar{W}$, with $\delta=0$ if $F_X(W)$ is a merge, and $\delta=1$ if $F_X(W)$ is a split. Notice that (\ref{eq:sevenequiv}) is a special case of this.

We now form the resulting quotient category and denote it by $\Z\dChrqu{X}(B,\dot{B})$. Notice that we get an induced functor
\begin{equation}\label{eq:abandon}
F_X\colon \Z\dChrqu{X}(B,\dot{B}) \to \Z\dChrq(\emptyset).
\end{equation}
\begin{remark}
This functor is not faithful. For example, consider a surgery diagram that closes with $X$ to a surgery diagram of type X or Y. For the morphisms $W$ and $\bar{W}$ arising from the surgery diagram we get $F_X(W) = 0 = F_X(\bar{W})$. However, neither $W$ nor $\bar{W}$ need to be zero morphisms.
\end{remark} 

Given a tangle diagram $\mathcal{D}$ in a $B$ and an $X$ as above, we can combine $\mathcal{D}$ and $X$ to a link diagram $\mathcal{D}_X$. Choosing arrows on the crossings of $\mathcal{D}$ and a sign assignment, we can now form a chain complex $[\mathcal{D}]$ over $\Z\dChrqu{X}(B,\dot{B})$.

\begin{example}
Let $\mathcal{D}$ be the tangle diagram from Figure \ref{fig:example}, with $B$ a disc.

\begin{figure}[ht]
\begin{tikzpicture}
\draw[ultra thick] (2, 1) ellipse (2.5 and 1);
\draw[thick] (0.05,1.65) -- (0.7, 1);
\draw[thick] (0.9,0.8) -- (1.7, 0);
\draw[thick] (2.3,0) -- (3.95, 1.65);
\draw[thick] (3.8,0.3) -- (3.3, 0.8);
\draw[thick] (3.1,1) -- (2.9, 1.2) to [out = 135, in = 45] (1.1, 1.2) -- (0.2,0.3);
\draw[->] (3,0.9) -- (3.4, 0.9);
\draw[->] (0.6, 0.9) -- (1, 0.9);
\end{tikzpicture}
\caption{\label{fig:example}A tangle consisting of two crossings.} 
\end{figure}
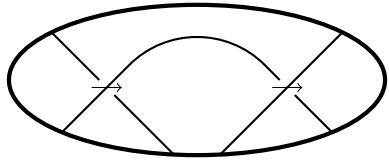

Given $X$, we can find a sign assignment so that the resulting cochain complex $[\mathcal{D}]$ is of the form
\[
\begin{tikzpicture}
\draw[thick] (0,1) to [out = 45, in = 315] (0,1.6);
\draw[thick] (0.4,1) to [out = 135, in = 225] (0.4, 1.5) to [out = 45, in = 135] (0.8, 1.5) to [out = 315, in = 45] (0.8, 1);
\draw[thick] (1.2, 1) to [out = 135, in = 225] (1.2, 1.6);
\draw[->] (1.4,1.4) -- (3.4, 2.3);
\draw[->] (1.4, 1.2) -- (3.4, 0.3);
\draw[thick] (3.6, 2) -- (3.7,2.1) to [out = 45, in = 135] (3.9,2.1) -- (4, 2);
\draw[thick] (3.6, 2.6) -- (3.7, 2.4) to [out = 315, in = 225] (3.9, 2.4) -- (4,2.5) to [out = 45, in = 135] (4.4, 2.5) to [out = 315, in = 45] (4.4, 2);
\draw[thick] (4.8,2) to [out = 135, in = 225] (4.8, 2.6);
\draw[thick] (3.6,0) to [out = 45, in = 315] (3.6, 0.6);
\draw[thick] (4, 0) to [out = 135, in = 225] (4, 0.5) to [out = 45, in = 135] (4.4, 0.5) -- (4.5, 0.4) to [out = 315, in = 225] (4.7, 0.4) -- (4.8, 0.6);
\draw[thick] (4.4,0) -- (4.5, 0.1) to [out = 45, in = 135] (4.7,0.1) -- (4.8,0);
\node[scale = 0.7] at (5.1,0) {$\{+1\}$};
\node[scale = 0.7] at (5.1,2) {$\{+1\}$};
\draw[->] (5.2, 0.3) -- node [scale = 0.9, sloped, above] {$(-1)^{\varepsilon_X}$} (7.2, 1.2);
\draw[->] (5.2, 2.3) -- (7.2, 1.4);
\draw[thick] (7.4, 1) -- (7.5, 1.1) to [out = 45, in = 135] (7.7, 1.1) -- (7.8,1);
\draw[thick] (7.4, 1.6) -- (7.5, 1.4) to [out = 315, in = 225] (7.7, 1.4) -- (7.8, 1.5) to [out = 45, in = 135] (8.2, 1.5) -- (8.3, 1.4) to [out = 315, in = 225] (8.5, 1.4) -- (8.6, 1.6);
\draw[thick] (8.2, 1) -- (8.3, 1.1) to [out = 45, in = 135] (8.5, 1.1) -- (8.6, 1);
\node[scale = 0.7] at (8.9, 1) {$\{+2\}$};
\end{tikzpicture}
\]
The number in curly parentheses indicates a change in $q$-degree. The morphisms are the obvious saddle morphisms, with only one having a possible non-positive sign. The value $\varepsilon_X$ does depend on $X$. If we choose $X$ so that $\mathcal{D}_X$ is the diagram of a Hopf link, we get a type C.3 face. Changing one of the arrows on a crossing turns this into a type A.3 face. We can also choose $X$ to give us unknots leading to type A.2 or C.2.
\end{example}

\section{Sign assignments}\label{sec:signs}
An obvious problem with trying to mimick Bar-Natan's scanning algorithm in odd Khovanov homology is that the sign assignment required in the complex is not local. We want to take a closer look at how sign assignments are obtained, and use this to get a specific one which will be well adapted for our purposes.

Let $Q=[0,1]^n$ be the CW-complex obtained from the obvious CW-structure on $[0,1]$ using cartesian products. The $0$-cells correspond to the vertices, the $1$-cells to the edges, and the $2$-cells to the various faces of the hypercube described in Section \ref{sec:recap}.

In \cite{MR3071132} a cochain $\varphi\in C^2(Q;\Z/2\Z)$ is constructed by sending a $2$-cell to $0$ if the corresponding face is of type A or X, and to $1$ otherwise. By [ORS, Lm.2.1] this is a cocycle, and since $Q$ is contractible, there has to be a cochain $\varepsilon\in C^1(Q;\Z/2\Z)$ with $\delta \varepsilon = \varphi$, which is exactly the condition that $\varepsilon$ is a sign assignment.

To make this more constructive, consider a chain homotopy $H_\ast \colon C_\ast(Q;\Z/2\Z) \to C_{\ast+1}(Q;\Z/2\Z)$ between the identity and a constant map on $Q$. We can then simply choose $\varepsilon = \varphi \circ H_1$. Let us construct $H_1$ explicitly.

As mentioned before, we can identify a $0$-cell in $Q$ with $c=(c_1,\ldots,c_n)\in \{0,1\}^n$. Given such $c$ and $i\in \{1,\ldots,n\}$, let us write
\[
c^i = (c_1,\ldots,c_{i-1},0,\ldots,0)\in \{0,1\}^{n-1}.
\]
Recall the notation $e^i_c$ from Section \ref{sec:recap}, which we now use to describe the $1$-cells of $Q$.

We can think of $C_0(Q;\Z/2\Z)$, resp.\ $C_1(Q;\Z/2\Z)$, as freely generated by $c\in  \{0,1\}^n$, resp.\ by the $e_c^i$. If we define $H_0\colon C_0(Q;\Z/2\Z) \to C_1(Q;\Z/2\Z)$ by
\[
H_0(c) = \sum_{i=1}^n c_i e_{c^i}^i,
\]
then $\partial H_0(c) = c + (0,\ldots,0)$. We can think of $H_0(c)$ as a cellular path between $c$ and $(0,\ldots,0)$ obtained by changing the rightmost coordinate with entry $1$ to $0$, and continuing until all entries are $0$.

To describe $H_1$, let us think of $C_2(Q;\Z/2\Z)$ as freely generated by the faces $f_c^{ij}$ described in Section \ref{sec:recap}. Also let 
\[
c^{j-1}=(c_1,\ldots,c_{j-2},0,\ldots,0)\in \{0,1\}^{n-2},
\]
so that
\[
f_{c^{j-1}}^{ij} = (c_1,\ldots,c_{i-1},\ast,c_i,\ldots, c_{j-2},\ast,0,\ldots, 0).
\]
Now define
\begin{equation}\label{eq:AgeOne}
H_1(e_c^i) = \sum_{j=i+1}^n c_{j-1} f_{c^{j-1}}^{ij}.
\end{equation}

\begin{lemma}
We have
\[
H_0\partial + \partial H_1 = \id_{C_1}.
\]
\end{lemma}

\begin{proof}
We need to check that $H_0\partial e^i_c+\partial H_1 e^i_c = e^i_c$ for every basis element $e_c^i$. If we write
\[
e_c^i = (c_1,\ldots,c_{i-1},\ast,c_i,\ldots,c_{n-1}),
\]
we get $\partial e^i_c = (c_1,\ldots,c_{i-1},1,c_i,\ldots,c_{n-1}) + (c_1,\ldots,c_{i-1},0,c_i,\ldots,c_{n-1})$, and
\begin{align*}
H_0\partial e^i_c = (c_1,\ldots,c_{i-1},\ast,0,\ldots,0) &+ \sum_{j = i}^{n-1} c_j (c_1,\ldots,c_{i-1},1,c_i,\ldots,c_{j-1},\ast,0,\ldots,0)\\
&+ \sum_{j = i}^{n-1} c_j (c_1,\ldots,c_{i-1},0,c_i,\ldots,c_{j-1},\ast,0,\ldots,0)
\end{align*}
From (\ref{eq:AgeOne}) we see that the two sums also appear in $\partial H_1(e^i_c)$. The remaining two sums in $\partial H_1(e^i_c)$ are given by
\[
\sum_{j=i}^{n-1} c_j (c_1,\ldots,c_{i-1},\ast,c_i,\ldots,c_{j-1},\varepsilon, 0,\ldots,0),
\]
for $\varepsilon = 0,1$. Notice that 
\[
c_j (c_1,\ldots,c_{i-1},\ast,c_i,\ldots,c_{j-1},1, 0,\ldots,0) = c_j (c_1,\ldots,c_{i-1},\ast,c_i,\ldots,c_j, 0,\ldots,0),
\] 
so we get
\begin{align*}
H_0\partial e^i_c  + \partial H_1 e^i_c &+ (c_1,\ldots,c_{i-1},\ast,0,\ldots,0) \\
= & \sum_{j=i}^{n-1} c_j (c_1,\ldots,c_{i-1},\ast,c_i,\ldots,c_{j}, 0,\ldots,0) \\
& + \sum_{j=i}^{n-1} c_j (c_1,\ldots,c_{i-1},\ast,c_i,\ldots,c_{j-1}, 0,\ldots,0)
\end{align*}
Now let 
\[
K = \{ j\in \{i,\ldots,n-1\}\, | \, c_j = 1\}.
\]
We set $k_{\min} = \min K$ and $K_1 = K \cup \{n\} -\{k_{\min}\}$. Then
\begin{align*}
H_0\partial e^i_c  + \partial H_1 e^i_c + (c_1,\ldots,c_{i-1},\ast,0,\ldots,0) = & \sum_{j\in K_1} e^i_{c^j} + \sum_{j\in K} e^i_{c^j} \\
= & \,e^i_{c^n} + e^i_{c^{k_{\min}}}.
\end{align*}
Since $e^i_{c^n} = e^i_c$ and $e^i_{c^{k_{\min}}} = (c_1,\ldots,c_{i-1},\ast,0,\ldots,0)$, the result follows.
\end{proof}

Using the standard induction proof for the existence of a chain contraction in a free, bounded below chain complex, we can extend $H_1$ to the required chain homotopy $H_\ast$. To sum up, we get the following sign assignment for a link diagram.

\begin{lemma}\label{lm:signassign}
Let $\mathcal{D}$ be a link diagram with $n$ crossings, where every crossing has a chosen arrow. Then $\varepsilon$ defined by
\[
\varepsilon(e_c^i) = \sum_{j = i+1}^n c_{j-1}\varphi(f_{c^{j-1}}^{ij}) \in \Z/2\Z
\]
is a sign assignment, where $\varphi(f)=0$ whenever the face $f$ is of type A or X in the hypercube, and $1$ otherwise.\hfill \qed
\end{lemma}

\begin{remark}
We get $\varepsilon(e)=0$ if 
\[
e = (c_1,\ldots,c_{i-1},\ast,0,\ldots,0).
\]
More generally, for $c\in \{0,1\}^{n-1}$ and $i\in \{1,\ldots,n\}$, we need to consider the faces of the form
\[
f = (c_1,\ldots,c_{i-1},\ast, c_i,\ldots, c_{j-2},\ast,0,\ldots,0),
\]
with $c_{j-1} = 1$ when calculating $\varepsilon(e_c^i)$.
\end{remark}

\begin{definition}
Let $\mathcal{S}$ be the set of all triples $(S,I_0,I_1)$, where $S$ is a closed smooth submanifold of $\R^2$, and $I_0, I_1$ are disjoint compact oriented smooth intervals embedded in $\R^2$, which intersect $S$ transversally in four points, and exactly at the endpoints of $I_0$ and $I_1$. 
\end{definition}

We want to think of $\varphi$ as a function $\varphi\colon \mathcal{S}\to \Z/2\Z$ rather than a cocycle $\varphi\in C^2(Q;\Z/2\Z)$. Each triple in $\mathcal{S}$ falls uniquely into a type A, C, X, or Y as in Figure \ref{fig:subcases}, allowing orientation preserving diffeomorphisms. So we set $\varphi(S,I_0,I_1) = 0$, if the triple is of type $A$ or $X$, and $1$ if it is of type $C$ or $Y$. Notice that $S$ can have many components, but only the ones intersecting the intervals determine the type.

\section{A substitute for the tensor product}

The scanning algorithm of \cite{MR2320156} can be roughly described as follows. Given a link diagram $\mathcal{D}$ with $n$ crossings, choose an ordering of the crossings, and form cochain complexes $C^\ast_i$ generated by the two smoothings with the saddle between them as the boundary. The first $i$ crossings form a tangle $T_i$ embedded in a compact $2$-manifold $B_i$. Start with $D^\ast_1 = C^\ast_1$, and assuming that $D^\ast_i$ exists for some $i\in \{1,\ldots, n-1\}$, we can form $D^\ast_{i+1}$ by
\begin{enumerate}
\item Form the cochain complex $E^\ast_{i+1} = D^\ast_i \otimes C^\ast_{i+1}$.
\item Deloop $E^\ast_{i+1}$, by replacing any circle component in an object with two objects where this component is removed, resulting in an isomorphic cochain complex $F_{i+1}^\ast$.
\item Perform as many Gaussian eliminations on $F_{i+1}^\ast$ as possible, and call the final result $D^\ast_{i+1}$.
\end{enumerate}
Repeating these three steps leads to a cochain complex $D_n^\ast$ chain homotopy equivalent to the Khovanov complex of the original diagram.

It is worth pointing out that if we only perform the steps (1) and (2), the complex at the end is the Khovanov complex (with a specific sign assignment coming from the tensor product), and if we only perform step (1) each time, the result is a cochain complex over an additive category such that the Khovanov complex is the result of applying a TQFT.

Delooping and Gaussian elimination does not present a difficulty in the odd situation.  Indeed, delooping in the odd case is already described in \cite[Prop.11.4]{MR3363817}, and Gaussian elimination \cite[Lm.3.2]{MR2320156} is a statement about cochain complexes in additive categories.

The tensor product represents a bit of a difficulty though, since our sign assignments do not follow a nice pattern that arises from the usual sign rules in a tensor product cochain complex. To rectify this, let us take a closer look at $E^\ast_{i+1}=D^\ast_i \otimes C^\ast_{i+1}$. The complex $D^\ast_i$ can be fairly large, but $C^\ast_{i+1}$ only has two objects, connected by a saddle morphism. We can visualize this as
\[
\begin{tikzpicture}
\node at (0,1.5) {$D^{k-1}_i\otimes C^0_{i+1}$};
\node at (3,1.5) {$D^k_i\otimes C^0_{i+1}$};
\node at (6,1.5) {$D^{k+1}_i\otimes C^0_{i+1}$};
\node at (3,0) {$D^{k-1}_i\otimes C^1_{i+1}$};
\node at (6,0) {$D^k_i\otimes C^1_{i+1}$};
\node at (9,0) {$D^{k+1}_i\otimes C^1_{i+1}$};
\node at (-1.4,1.5) {$\cdots$};
\node at (10.4,0) {$\cdots$};
\node at (0.7,0) {$\cdots$};
\draw[->] (1,1.5) -- (2.1,1.5);
\draw[->] (1,1.4) -- node [above, scale = 0.7, sloped] {$(-1)^{k-1}$}  (2,0.1);
\draw[->] (1,0) -- (2,0);
\draw[->] (3.9,1.5) -- (5,1.5);
\draw[->] (3.9,1.4) -- node [above, scale = 0.7, sloped] {$(-1)^k$} (5.1,0.1);
\draw[->] (4,0) -- (5.1,0);
\draw[->] (7,1.5) -- (8,1.5);
\draw[->] (7,1.4) -- node [above, scale = 0.7, sloped] {$(-1)^{k+1}$} (8,0.1);
\draw[->] (6.9,0) -- (8,0);
\node at (8.4,1.5) {$\cdots$};
\end{tikzpicture}
\]
The cochain complexes $D^\ast_i$ are over additive categories $\mathfrak{C}_i$ which are similar to categories $\Z\dChrqu{X_i}(B_i,\dot{B}_i)$ with the $\dot{B}_i$ the boundary of the tangle $T_i$. The $X_i$, which give a closure to the tangle $T_i$ are not needed in \cite{MR2320156} . We can think of the top horizontal line as the result of applying a functor $F\colon \mathfrak{C}_i\to \mathfrak{C}_{i+1}$ to $D^\ast_i$, and the bottom horizontal line as the result of applying a different functor $G\colon \mathfrak{C}_i\to \mathfrak{C}_{i+1}$ to $D^\ast_i$. The diagonal maps can then be thought of as coming from a natural transformation $\sigma$ between these two functors. 
In this setting the cochain complex $E^\ast_i$ is a mapping cone of the natural transformation.

The next definition formalizes this concept.

\begin{definition}\label{def:cone}
Let $\mathfrak{C}$, $\mathfrak{D}$ be additive categories, let $F, G\colon \mathfrak{C}\to\mathfrak{D}$ be functors, and let $\sigma$ be a natural transformation between $F$ and $G$. Given a cochain complex $C^\ast$ over $\mathfrak{C}$, define the {\em mapping cone} $\mathcal{C}^\sigma_{F,G}(C)^\ast$ over $\mathfrak{D}$ as follows. We set
\[
\mathcal{C}^\sigma_{F,G}(C)^n = F(C^n)\oplus G(C^{n-1})
\]
and let the coboundary be represented by the matrix
\[
\delta^n_\mathcal{C} = \begin{pmatrix} F(\delta^n) & 0 \\ \sigma_{C^n} & -G(\delta^{n-1}) \end{pmatrix}\colon F(C^n)\oplus G(C^{n-1}) \to F(C^{n+1})\oplus G(C^{n}).
\]
\end{definition}
It is easy to see that the cone behaves well with cochain maps and chain homotopies. Hence if $C^\ast$ and $D^\ast$ are chain homotopy equivalent, so are $\mathcal{C}^\sigma_{F,G}(C)^\ast$ and $\mathcal{C}^\sigma_{F,G}(D)^\ast$.

\begin{remark}
To recover the tensor product with Definition \ref{def:cone}, we should have used $G(\delta^{n-1})$ and $(-1)^n\sigma_{C^n}$ in the matrix of the coboundary. However, the given choice will fit in better with our setting.
\end{remark}

\section{The Algorithm}\label{sec:algorithm}

Let $\mathcal{D}$ be an oriented link diagram with $n$ crossings, and choose an arrow at every crossing. Now choose an ordering of the crossings $c_1,\ldots, c_n$. We want to define compact $2$-manifolds $B_1,\ldots, B_n$ embedded into the plane. First let $A_i$ be a small disc around the crossing $c_i$ for each $i=1,\ldots,n$, so small that they are all disjoint.

Let $B_1 = A_1$, and assume that we have defined $B_i$ for some $i\in \{1,\ldots,n-1\}$. Let $B_{i+1}'$ be the disjoint union of $B_i$ and $A_{i+1}$, and consider the four points of $\partial A_{i+1}\cap \mathcal{D}$. Each of these points $a_1,\ldots,a_4$ is part of an arc between the crossing $c_{i+1}$ and another crossing $c_{i_j}$ with $j\in \{1,\ldots, 4\}$. For each $j\in \{1,\ldots,4\}$ with $i_j \leq i$, we add a small thickening of this arc to $B_{i+1}'$ and call the result $B_{i+1}''$. This is a surface of genus $0$ with several discs removed. We obtain $B_{i+1}$ by re-inserting those bounded discs which do not contain crossings. See Figure \ref{fig:exampleB} for an example.

\begin{figure}[ht]
\begin{tikzpicture}[remember picture]
\draw[color=gray, ultra thick, fill = gray!20] (0.5,-1.2) to [out = 90, in = 270] (-0.2, 0) to [out = 90, in = 0] (-0.8,0.6) to [out = 180, in = 90] (-1.8, -0.8) to [out = 270, in = 180] (-0.5, -1.6) to [out = 0, in = 270] (0.5, -1.2);
\node[inner sep=0] at (0,0)
{\includegraphics[width=3cm,height=3cm]{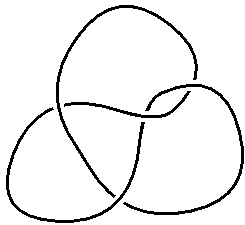}};
\draw[color=gray, ultra thick] (0.5,-1.2) to [out = 90, in = 270] (-0.2, 0) to [out = 90, in = 0] (-0.8,0.6) to [out = 180, in = 90] (-1.8, -0.8) to [out = 270, in = 180] (-0.5, -1.6) to [out = 0, in = 270] (0.5, -1.2);
\draw[color = gray, ultra thick, fill = white] (-0.9, -0.85) circle (0.3);
\draw[color=gray, ultra thick, fill = gray!20] (5.5,-1.2) to [out = 90, in = 270] (4.8, 0) to [out = 90, in = 0] (4.2,0.6) to [out = 180, in = 90] (3.2, -0.8) to [out = 270, in = 180] (4.5, -1.6) to [out = 0, in = 270] (5.5, -1.2);
\node[inner sep=0] at (5,0)
{\includegraphics[width=3cm,height=3cm]{img_fig8.png}};
\draw[color=white, ultra thick, fill = white] (5.5,-1.2) to [out = 90, in = 270] (4.8, 0) to [out = 90, in = 0] (4.2,0.6) -- (4.2,1.45) -- (6.5,1.45) -- (6.5,-1.5) -- (5.5, -1.5) -- (5.5, -1.2);
\draw[very thick] (5.5, -1.35) to [out = 0, in = 270] (6, -0.6) to [out = 90, in = 0] (5.6, -0.2) to [out = 180, in = 90] (5.2, -0.6);
\draw[very thick] (4.8, 0.09) to [out = 350, in = 180] (5.6, 0) to [out = 0, in = 270] (6, 0.6) to [out = 90, in = 0] (5, 1.2) to [out = 180, in = 80] (4.27, 0.6);
\draw[color=gray, ultra thick] (5.5,-1.2) to [out = 90, in = 270] (4.8, 0) to [out = 90, in = 0] (4.2,0.6) to [out = 180, in = 90] (3.2, -0.8) to [out = 270, in = 180] (4.5, -1.6) to [out = 0, in = 270] (5.5, -1.2);
\end{tikzpicture}
\caption{\label{fig:exampleB}The shaded region on the left is $B_2''$, and on the right we have $B_2$. The diagram on the right is $\mathcal{D}_2$, with $X_2$ outside the shaded region.}
\end{figure}

We also set $\dot{B}_i = \partial B_i\cap \mathcal{D}$ for $i=1,\ldots, n$. To form the embedded compact $1$-manifold $X_i$ for $i=1,\ldots,n$, let $\mathcal{D}_i$ be the link diagram obtained from $\mathcal{D}$ by smoothing each crossing $c_j$ with $j > i$ with the $0$-smoothing. Then set $X_i = \mathcal{D}_i\cap (\R^2-(B_i-\partial B_i))$.

With these choices we can define a functor 
\[
F_i\colon \Z\dChrqu{X_i}(B_i,\dot{B}_i)\to \Z\dChrqu{X_{i+1}}(B_{i+1},\dot{B}_{i+1})
\]
as follows. Given an object $(S,q)$ with $S$ a  smoothing embedded in $B_i$, we can combine it with the $0$-smoothing of the crossing $c_{i+1}$ and all of the arcs around $c_{i+1}$ which are contained in $B_{i+1}$. Calling this new smoothing $S_0$, we set $F_i((S,q)) = (S_0,q)$. On morphisms, we extend a chronological cobordism by a cylinder of the new part.

Notice that $S_0 - S = X_i \cap B_{i+1}$, so the functor preserves all the relations in the quotient categories.

Before we even define $G_i$, let us denote what is going to be the natural transformation $\sigma^i$ between $F_i$ and $G_i$. If $(S,q)$ is an object in $\Z\dChrqu{X_i}(B_i,\dot{B}_i)$, let $S_1$ be the smoothing obtained from $S_0$ by performing the surgery coming from the crossing $c_{i+1}$. Then $\sigma^i_{(S,q)}\in \Mor((S_0,q),(S_1,q+1))$ is the morphism coming from the obvious saddle between $S_0$ and $S_1$.

To define the functor
\[
G_i\colon \Z\dChrqu{X_i}(B_i,\dot{B}_i)\to \Z\dChrqu{X_{i+1}}(B_{i+1},\dot{B}_{i+1})
\]
we set $G_i(S,q) = (S_1,q+1)$ on objects, but on morphisms we need to be a bit more careful with signs, so that $\sigma^i$ is indeed a natural transformation between these functors.

So let $W$ be a morphism between objects $(S,q)$ and $(S',p)$ in $\Z\dChrqu{X_i}(B_i,\dot{B}_i)$ which is basic in the sense that it is one of the morphisms listed in Figure \ref{fig:basicCob}, after combining it with $X_i\times [0,1]$. Indeed, let us write $W^i$ for this combination.

If $W^i$ is a merge or a split, consider $S_0$, the smoothing used in the definition of $F_i$, and combine it to a closed $1$-manifold $S^i_0$ with $X_{i+1}$. Let $I_0$ be the oriented surgery arc from the crossing $c_{i+1}$, and $I_1$ the oriented surgery arc representing the merge or split of $W^i$. Then $(S^i_0, I_0, I_1)\in \mathcal{S}$, and we set
\begin{equation}\label{eq:signsforcob}
G_i(W) = (-1)^{1+\varphi(S^i_0,I_0,I_1)} W_1,
\end{equation}
where $W_1$ is the morphism obtained from $W$ by combining it with the cylinder of the $1$-smoothing of the crossing $c_{i+1}$.

If $W$ is a birth morphism, and we set $G_i(W) = W_1$, where $W_1$ is again the combination of $W$ with a cylinder as before.

If $W$ is a death morphism or a dotted morphism, we set
\begin{equation}\label{eq:signfordot}
G_i(W) = (-1)^\delta W_1,
\end{equation}
where $\delta = 0$, if the saddle morphism induced by the crossing $c_{i+1}$ leads to a merge from $S_0^i$ to $S_1^i$, and $\delta = 1$, if it leads to a split from $S_0^i$ to $S_1^i$.

The signs accompanying the morphisms are chosen so that $\sigma^i$ is indeed a natural transformation between $F_i$ and $G_i$ (note that in (\ref{eq:signsforcob}) $\varphi$ ensures anti-commutativity, so that the extra $(-1)$ gives commutativity), although we still need to check that $G_i$ is in fact well defined. Recall that $\Z\dChrqu{X_i}(B_i,\dot{B}_i)$ is a quotient using an equivalence relation $\sim_i$, while $\Z\dChrqu{X_{i+1}}(B_{i+1},\dot{B}_{i+1})$ is a quotient using another equivalence relation $\sim_{i+1}$. 

\begin{theorem}
The definition of $G_i$ above gives rise to a well defined functor
\[
G_i\colon \Z\dChrqu{X_i}(B_i,\dot{B}_i)\to \Z\dChrqu{X_{i+1}}(B_{i+1},\dot{B}_{i+1}),
\]
and $\sigma^i$ is a natural transformation between $F_i$ and $G_i$.
\end{theorem}

\begin{proof}
We have already seen that $\sigma^i$ is a natural transformation, provided that $G_i$ is well defined. So we need to check that if $W\sim_i W'$, then $G_i(W)\sim_{i+1}G_i(W')$.

Relations of the form (\ref{eq:firstequiv}) and (\ref{eq:neckcut}) have to take place away from the crossing $c_{i+1}$ and are therefore preserved. Similarly the relations in (\ref{eq:sevenequiv}) are easily seen to be preserved. The first relation in (\ref{eq:eightequiv}) is slightly more interesting: let $W$ represent the split in this relation. Then $G_i(W)$ can come with a minus sign, depending on whether the saddle coming from the crossing $c_{i+1}$ leads to a split or not. But then $G_i$ of the following death comes with the same factor, so the relation is preserved. The relations in (\ref{eq:ninthequiv}) are preserved with similar arguments.

Let us now consider a surgery diagram in $B_i$ involving two surgeries, and let $W$, $\bar{W}$ be the two chronological cobordisms arising. Choose $\delta, \delta_1 \in \{0,1\}$ such that 
\[
W\sim_i (-1)^\delta \bar{W} \mbox{ and } W_1\sim_{i+1} (-1)^{\delta_1} \bar{W}_1. 
\]
We can write $W = W^+\circ W^-$ with both $W^-$ and $W^+$ saddle cobordisms, and similarly we write $\bar{W} = \bar{W}^+\circ \bar{W}^-$. Furthermore, for a saddle cobordism $\tilde{W}$ let $T_{\tilde{W}}\in \mathcal{S}$ be the triple such that $G_i(\tilde{W}) = (-1)^{1+\varphi(T_{\tilde{W}})} \tilde{W}_1$. Also, $W$ and $\bar{W}$ give rise to a triple $T\in\mathcal{S}$ with $\delta = \varphi(T)$, while $W_1$ and $\bar{W}_1$ give rise to a triple $T_1$ with $\delta_1 = \varphi(T_1)$.

By \cite[Lm.2.1]{MR3071132} we get
\[
\varphi(T)+\varphi(T_1)+\varphi(T_{W^+}) + \varphi(T_{W^-}) + \varphi(T_{\bar{W}^+}) + \varphi(T_{\bar{W}^-}) = 0 \in \Z/2\Z.
\]
This means that
\[
\delta_1 = \delta + \varphi(T_{W^+}) + \varphi(T_{W^-}) + \varphi(T_{\bar{W}^+}) + \varphi(T_{\bar{W}^-}),
\]
with the last four summands needed in the definition of $G_i(W)$ and $G_i(\bar{W})$. As a result, we get $G_i(W)\sim_{i+1} G_i((-1)^\delta \bar{W})$, exactly as was required.

The cases where we slide a birth or a death past another critical point are easier. If this other critical point has index $1$ we need to be slightly careful if the saddle coming from the crossing $c_{i+1}$ changes between a merge and a split, but again the signs in $G_i$ make this work.

%

The cases in (\ref{eq:fifthequiv}), where we slide dots past births, deaths, or other dots, are similar to the cases where dots are replaced with deaths. 

Finally, if $W$ is a framed saddle and $\bar{W}$ the same saddle with opposite framing. Again it is possible that the crossing $c_{i+1}$ results in a change between $W$ being part of a merge or split, but again this is being corrected by $G_i$.
\end{proof}

We can now mimick the scanning algorithm of \cite{MR2320156}. Form the cochain complex $D^\ast_1$ generated by the two smoothings with the oriented saddle from the first crossing $c_1$ between them as coboundary. This is a cochain complex over $\Z\dChrqu{X_1}(B_1,\dot{B}_1)$. For $i\geq 1$ define
\[
D^\ast_{i+1} = \mathcal{C}^{\sigma^i}_{F_i,G_i}(D_i)^\ast,
\]
a cochain complex over $\Z\dChrqu{X_{i+1}}(B_{i+1},\dot{B}_{i+1})$. The final complex $D_n^\ast$ can be treated as a complex over $\Z\dChrq(\emptyset)$. Furthermore, since we have not done any cancellations yet, all the morphisms in $D_n^\ast$ are saddles, so we only ever used (\ref{eq:signsforcob}) when forming the $G_i(D^\ast_i)$. Also, the $(-1)$ in the coboundary of the mapping cone removes the extra $(-1)$ in (\ref{eq:signsforcob}), so $D_n^\ast$ gives rise to the odd Khovanov complex with the sign assignment from Lemma \ref{lm:signassign} upon applying the functor $F$ from Lemma \ref{lm:functor}.

As mentioned before, delooping and Gaussian elimination work as in \cite{MR2320156}, so we have a scanning algorithm for odd Khovanov homology.

\section{Practicalities and Computations}

In order to implement the algorithm from Section \ref{sec:algorithm} let us take a closer look at delooping and how it is used in practice.

\begin{lemma}
Let $S$ be a circle embedded in $\R^2$. In $\Z\dChrq(\emptyset)$ we have 
\[
(S,q) \cong (\emptyset,q-1)\oplus (\emptyset,q+1).
\]
Furthermore, the isomorphisms are given by
\[
\begin{pmatrix}
\begin{tikzpicture}
\node at (0.2,0.3) { };
\bowlud{0}{0}{0.5}{0.1}
\end{tikzpicture}
\\
\begin{tikzpicture}
\node at (0.2,0.2) { };
\bowlud{0}{0}{0.5}{0.1}
\node at (0.2,0.16) [scale = 0.5] {$\bullet$};
\end{tikzpicture}
\end{pmatrix}
\colon (S,q) \longrightarrow (\emptyset,q-1) \oplus (\emptyset,q+1)
\]
and
\[
\pushQED{\qed}
\begin{pmatrix}
\begin{tikzpicture}
\node at (0.2,0.1) { };
\bowl{0}{0}{0.5}{0.1}
\node at (0.2,-0.16) [scale = 0.5] {$\bullet$};
\end{tikzpicture}
&
\begin{tikzpicture}
\bowl{0}{0}{0.5}{0.1}
\end{tikzpicture}
\end{pmatrix}
\colon (\emptyset,q-1) \oplus (\emptyset,q+1) \longrightarrow (S,q).\qedhere
\popQED
\]
\end{lemma} 

The proof of \cite[Lm.3.1]{MR2320156} works without changes. Indeed, it follows straight from (\ref{eq:firstequiv}) and (\ref{eq:neckcut}). Furthermore, it is local, so works for any circle component in objects over $\Z\dChrqu{X}(B,\dot{B})$.

Let us check what happens to the morphisms
\[
\begin{tikzpicture}
\draw[very thick] (0,0) to [out = 45, in = 315] (0,0.6);
\draw[very thick] (0.6,0.3) circle (0.25);
\draw[->] (1.2,0.3) -- (2.5,0.3);
\draw[thick] (1.6, 0.4) to [out = 45, in = 315] (1.6, 0.8);
\draw[thick] (2, 0.6) circle (0.18);
\draw (1.68,0.6) -- (1.82,0.6);
\relsplit{2.8}{0}{1}{very thick}
\node[scale = 0.75] at (4,0) {$\{+1\}$};
\node at (5,0.3) {and};
\draw[very thick] (6,0) -- (6.1,0.1) to [out = 45, in = 135] (6.5,0.1) to [out = 315, in = 225] (6.8,0.1) to [out = 45, in = 315] (6.8, 0.5) to [out = 135, in = 45] (6.5, 0.5) to [out = 225, in = 315] (6.1, 0.5) -- (6, 0.6);
\draw[->] (7.2,0.3) -- (8.5,0.3);
\draw[very thick] (7.6,0.4) -- (7.675,0.475) to [out = 45, in = 135] (7.975, 0.475) to [out = 315, in = 225] (8.2,0.475) to [out = 45, in = 315] (8.2, 0.775) to [out = 135, in = 45] (7.975, 0.775) to [out = 225, in = 315] (7.675, 0.775) -- (7.6, 0.85);
\draw[->] (7.825, 0.7) -- (7.825, 0.55);
\relmerge{8.8}{0}{1}{very thick}
\node[scale = 0.75] at (9.95,0) {$\{+1\}$}; 
\end{tikzpicture}
\]
after delooping. The first one turns into a morphism
$
\begin{tikzpicture}[baseline=0.5ex]
\draw[thick] (0,0) to [out = 45, in = 315] (0, 0.3);
\node[scale = 0.5] at (0.3,0.05) {$\{-1\}$};
\draw[->] (0.55,0.15) -- (0.85,0.15);
\draw[thick] (1,0) to [out = 45, in = 315] (1, 0.3);
\node[scale = 0.5] at (1.3,0.05) {$\{+1\}$};
\end{tikzpicture}
$ given by the movie presentation
\[
\begin{tikzpicture}
\draw[very thick] (0,0) to [out=45, in = 315] (0,0.5);
\birth{0.5}{0.25}{0.25}{very thick}
\draw[very thick] (1,0) to [out=45, in = 315] (1,0.5);
\draw[very thick] (1.5,0.25) circle (0.25);
\node at (1.5,0) {$\bullet$};
\draw[very thick] (2,0) to [out=45, in = 315] (2,0.5);
\draw[very thick] (2.5,0.25) circle (0.25);
\draw (2.25, 0.25) -- (2.1,0.25);
\draw[dashed] (-0.15,0.7) rectangle (2.95,-0.2);
\draw[dashed] (0.875,0.7) -- (0.875,-0.2);
\draw[dashed] (1.875,0.7) -- (1.875,-0.2);
\node at (3.3,0.25) {$\sim$};
\draw[very thick] (3.8,0) to [out=45, in = 315] (3.8,0.5);
\birth{4.3}{0.25}{0.25}{very thick}
\draw[very thick] (4.8,0) to [out=45, in = 315] (4.8,0.5);
\draw[very thick] (5.3,0.25) circle (0.25);
\draw (5.05, 0.25) -- (4.9,0.25);
\draw[very thick] (5.8,0) to [out=45, in = 315] (5.8,0.5);
\node at (5.91,0.25) {$\bullet$};
\draw[dashed] (3.65,0.7) rectangle (6.05,-0.2);
\draw[dashed] (4.675,0.7) -- (4.675,-0.2);
\draw[dashed] (5.675,0.7) -- (5.675,-0.2);
\node at (6.4, 0.25) {$\sim$};
\draw[very thick] (6.9,0) to [out=45, in = 315] (6.9,0.5);
\node at (7.01,0.25) {$\bullet$};
\draw[dashed] (6.75,0.7) rectangle (7.2,-0.2);
\end{tikzpicture}
\]
and a morphism $
\begin{tikzpicture}[baseline=0.5ex]
\draw[thick] (0,0) to [out = 45, in = 315] (0, 0.3);
\node[scale = 0.5] at (0.3,0.05) {$\{+1\}$};
\draw[->] (0.55,0.15) -- (0.85,0.15);
\draw[thick] (1,0) to [out = 45, in = 315] (1, 0.3);
\node[scale = 0.5] at (1.3,0.05) {$\{+1\}$};
\end{tikzpicture}
$ given by
\[
\begin{tikzpicture}
\draw[very thick] (3.8,0) to [out=45, in = 315] (3.8,0.5);
\birth{4.3}{0.25}{0.25}{very thick}
\draw[very thick] (4.8,0) to [out=45, in = 315] (4.8,0.5);
\draw[very thick] (5.3,0.25) circle (0.25);
\draw (5.05, 0.25) -- (4.9,0.25);
\draw[dashed] (3.65,0.7) rectangle (5.7,-0.2);
\draw[dashed] (4.675,0.7) -- (4.675,-0.2);
\node at (6.25,0.25) {$\sim \, 1$};
\end{tikzpicture}
\]
The second one turns into a morphism 
$
\begin{tikzpicture}[baseline=0.5ex]
\draw[thick] (0.3,0) to [out = 45, in = 315] (0.3, 0.3);
\draw[->] (0.55,0.15) -- (0.85,0.15);
\draw[thick] (1,0) to [out = 45, in = 315] (1, 0.3);
\end{tikzpicture}
$ given by the movie presentation
\[
\begin{tikzpicture}
\draw[very thick] (0,0) -- (0.1,0.1) to [out = 45, in = 135] (0.35,0.075) to [out = 315, in = 225] (0.65,0.075) to [out = 45, in = 315] (0.65, 0.425) to [out = 135, in = 45] (0.35, 0.425) to [out = 225, in = 315] (0.1, 0.4) -- (0, 0.5);
\draw[->] (0.2,0.35) -- (0.2,0.15);
\draw[very thick] (1,0) to [out = 45, in = 315] (1,0.5);
\death{1.5}{0.25}{0.2}{very thick}
\draw[dashed] (-0.15,0.7) rectangle (1.9,-0.2);
\draw[dashed] (0.875,0.7) -- (0.875,-0.2);
\node at (2.45,0.25) {$\sim \, 1$};
\end{tikzpicture}
\]
and a morphism 
$
\begin{tikzpicture}[baseline=0.5ex]
\draw[thick] (0.3,0) to [out = 45, in = 315] (0.3, 0.3);
\draw[->] (0.55,0.15) -- (0.85,0.15);
\draw[thick] (1,0) to [out = 45, in = 315] (1, 0.3);
\node[scale = 0.5] at (1.3,0.05) {$\{+2\}$};
\end{tikzpicture}
$ given by 
\[
\begin{tikzpicture}
\draw[very thick] (0,0) -- (0.1,0.1) to [out = 45, in = 135] (0.35,0.075) to [out = 315, in = 225] (0.65,0.075) to [out = 45, in = 315] (0.65, 0.425) to [out = 135, in = 45] (0.35, 0.425) to [out = 225, in = 315] (0.1, 0.4) -- (0, 0.5);
\draw[->] (0.2,0.35) -- (0.2,0.15);
\draw[very thick] (1,0) to [out = 45, in = 315] (1,0.5);
\draw[very thick] (1.5,0.25) circle (0.25);
\node at (1.5, 0) {$\bullet$};
\draw[very thick] (2,0) to [out = 45, in = 315] (2,0.5);
\death{2.5}{0.25}{0.2}{very thick}
\draw[dashed] (-0.15,0.7) rectangle (2.9,-0.2);
\draw[dashed] (0.875,0.7) -- (0.875,-0.2);
\draw[dashed] (1.875, 0.7) -- (1.875,-0.2);
\node at (3.4,0.25) {$\sim -$};
\draw[very thick] (4,0) -- (4.1,0.1) to [out = 45, in = 135] (4.35,0.075) to [out = 315, in = 225] (4.65,0.075) to [out = 45, in = 315] (4.65, 0.425) to [out = 135, in = 45] (4.35, 0.425) to [out = 225, in = 315] (4.1, 0.4) -- (4, 0.5);
\node at (4.5,0.02) {$\bullet$};
\draw[very thick] (5,0) -- (5.1,0.1) to [out = 45, in = 135] (5.35,0.075) to [out = 315, in = 225] (5.65,0.075) to [out = 45, in = 315] (5.65, 0.425) to [out = 135, in = 45] (5.35, 0.425) to [out = 225, in = 315] (5.1, 0.4) -- (5, 0.5);
\draw[->] (5.2,0.35) -- (5.2,0.15);
\draw[very thick] (6,0) to [out = 45, in = 315] (6,0.5);
\death{6.5}{0.25}{0.2}{very thick}
\draw[dashed] (3.85,0.7) rectangle (6.9,-0.2);
\draw[dashed] (4.875,0.7) -- (4.875, -0.2);
\draw[dashed] (5.875,0.7) -- (5.875,-0.2);
\node at (7.4,0.25) {$\sim -$};
\draw[very thick] (8,0) to [out=45, in = 315] (8,0.5);
\node at (8.11,0.25) {$\bullet$};
\draw[dashed] (7.85,0.7) rectangle (8.3,-0.2);
\end{tikzpicture}
\]
In particular, the split morphism leads to different signs for the two delooped morphisms after we cancel the critical points. 

\begin{example}
Consider the $(2,N)$-torus link with $N\geq 2$. We use the standard braid diagram closure. We order the crossings by going from bottom to top, and choose a right-pointing arrow at every crossing. Then
\[
D_1^\ast =
\begin{tikzpicture}[baseline=1.3ex]
\torusv{0}{0}{0.6}{0.6}{very thick}
\draw[->] (0.8,0.3) -- (1.8,0.3);
\torusv{1.1}{0.4}{0.4}{0.4}{very thick}
\draw[->] (1.15,0.6) -- (1.45,0.6);
\torush{2}{0}{0.6}{0.6}{very thick}
\node[scale = 0.7] at (2.9,0) {$\{+1\}$};
\end{tikzpicture}
\]
and
\[
D_2^\ast =
\begin{tikzpicture}[baseline=-2ex]
\torusv{0}{0}{0.6}{0.3}{very thick}
\torusv{0}{0.3}{0.6}{0.3}{very thick}
\draw[->] (0.8,0.3) -- node[above, scale = 0.7] {$+$} (1.8,0.3);
\draw[->] (0.8,0.2) -- node[above, scale = 0.7] {$+$} (1.8, -0.6);
\draw[->] (3.2,0.2) -- node[above, scale = 0.7] {$+$} (4.2,-0.6);
\draw[->] (3.2,-0.7) -- node[above, scale = 0.7] {$+$} (4.2,-0.7);
\torush{2}{0}{0.6}{0.3}{very thick}
\torusv{2}{0.3}{0.6}{0.3}{very thick}
\torusv{2}{-1}{0.6}{0.3}{very thick}
\torush{2}{-0.7}{0.6}{0.3}{very thick}
\torush{4.4}{-1}{0.6}{0.3}{very thick}
\torush{4.4}{-0.7}{0.6}{0.3}{very thick}
\node[scale = 0.7] at (2.9,0) {$\{+1\}$};
\node[scale = 0.7] at (2.9,-1) {$\{+1\}$};
\node[scale = 0.7] at (5.3,-1) {$\{+2\}$};
\end{tikzpicture}
\]
where each morphism is a saddle, and we only show the sign of it. Notice that only the lower horizontal line is coming from the functor $G_1$, and the relevant surgery diagram is of type A.3.

After delooping this turns into
\[
\begin{tikzpicture}
\torusv{-1}{0}{0.6}{0.6}{very thick}
\torush{2}{0}{0.6}{0.6}{very thick}
\torush{2}{-1}{0.6}{0.6}{very thick}
\torush{5.4}{0}{0.6}{0.6}{very thick}
\torush{5.4}{-1}{0.6}{0.6}{very thick}
\node[scale = 0.7] at (2.9,0) {$\{+1\}$};
\node[scale = 0.7] at (2.9,-1) {$\{+1\}$};
\node[scale = 0.7] at (6.3,-1) {$\{+1\}$};
\node[scale = 0.7] at (6.3,0) {$\{+3\}$};
\draw[->] (-0.2,0.3) -- node[above, scale = 0.7] {$+$} (1.8,0.3);
\draw[->] (-0.2,0.2) -- node[above, scale = 0.7] {$+$} (1.8, -0.6);
\draw[->] (3.2,0.2) -- node[above,sloped,scale = 0.7, near end] {$-1$} (5.2,-0.6);
\draw[->] (3.2,-0.7) -- node[above,scale=0.7] {$1$} (5.2,-0.7);
\draw[-, line width=6pt, draw=white] (3.2,-0.6) -- (5.2,0.2);
\draw[->] (3.2,-0.6) -- (5.2,0.2);
\draw[->] (3.2,0.3) -- (5.2,0.3); 
\torush{4.1}{0.4}{0.2}{0.2}{thick}
\node[scale = 0.5] at (4.21,0.54) {$\bullet$};
\torush{3.4}{-0.4}{0.2}{0.2}{thick}
\node[scale = 0.7] at (3.25,-0.3) {$-$};
\node[scale = 0.5] at (3.51, -0.34) {$\bullet$};
\end{tikzpicture}
\]
We can cancel the two generators in the lower row, to get
\[
E^\ast_2 = 
\begin{tikzpicture}[baseline=1.3ex]
\torusv{-1}{0}{0.6}{0.6}{very thick}
\torush{2}{0}{0.6}{0.6}{very thick}
\torush{5.4}{0}{0.6}{0.6}{very thick}
\draw[->] (-0.2,0.3) -- (1.8,0.3);
\draw[->] (3.2,0.3) -- (5.2,0.3);
\torusv{0.6}{0.4}{0.4}{0.4}{thick}
\draw[->] (0.65,0.6) -- (0.95,0.6);
\torush{3.6}{0.4}{0.4}{0.4}{thick}
\torush{4.4}{0.4}{0.4}{0.4}{thick}
\node[scale = 0.7] at (4.2,0.6) {$-$};
\node[scale = 0.7] at (3.81,0.68) {$\bullet$};
\node[scale = 0.7] at (4.61,0.52) {$\bullet$};
\node[scale = 0.7] at (2.9,0) {$\{+1\}$};
\node[scale = 0.7] at (6.3,0) {$\{+3\}$};
\end{tikzpicture}
\]
The next mapping cone is given by
\[
\begin{tikzpicture}
\torusv{-1}{0}{0.6}{0.3}{very thick}
\torusv{-1}{0.3}{0.6}{0.3}{very thick}
\torush{2}{0}{0.6}{0.3}{very thick}
\torusv{2}{0.3}{0.6}{0.3}{very thick}
\torush{5.4}{0}{0.6}{0.3}{very thick}
\torusv{5.4}{0.3}{0.6}{0.3}{very thick}
\torusv{2}{-1}{0.6}{0.3}{very thick}
\torush{2}{-0.7}{0.6}{0.3}{very thick}
\torush{5.4}{-1}{0.6}{0.3}{very thick}
\torush{5.4}{-0.7}{0.6}{0.3}{very thick}
\torush{8.8}{-1}{0.6}{0.3}{very thick}
\torush{8.8}{-0.7}{0.6}{0.3}{very thick}
\torush{3.85}{0.4}{0.2}{0.2}{thick}
\torusv{3.85}{0.6}{0.2}{0.2}{thick}
\torush{4.35}{0.4}{0.2}{0.2}{thick}
\torusv{4.35}{0.6}{0.2}{0.2}{thick}
\node[scale = 0.5] at (3.96,0.55) {$\bullet$};
\node[scale = 0.5] at (4.46,0.46) {$\bullet$};
\node[scale = 0.6] at (4.2,0.6) {$-$};
\torush{6.8}{-0.6}{0.2}{0.2}{thick}
\torush{6.8}{-0.4}{0.2}{0.2}{thick}
\torush{7.3}{-0.6}{0.2}{0.2}{thick}
\torush{7.3}{-0.4}{0.2}{0.2}{thick}
\node[scale = 0.5] at (6.91,-0.45) {$\bullet$};
\node[scale = 0.5] at (7.41,-0.54) {$\bullet$};
\node[scale = 0.6] at (7.15,-0.4) {$-$};
\draw[->] (-0.2,0.3) -- node [above, scale = 0.7] {$+$} (1.8,0.3);
\draw[->] (-0.2,0.2) -- node [above, scale = 0.7] {$+$} (1.8,-0.6);
\draw[->] (3.2,0.3) -- (5.2,0.3);
\draw[->] (3.2,0.2) -- node [above, scale = 0.7] {$+$} (5.2,-0.6);
\draw[->] (3.2,-0.7) -- node [above, scale = 0.7] {$+$} (5.2,-0.7);
\draw[->] (6.6,0.2) -- node [above, scale = 0.7] {$+$} (8.6,-0.6);
\draw[->] (6.6,-0.7) -- (8.6,-0.7);
\node[scale = 0.7] at (2.9,0) {$\{+1\}$};
\node[scale = 0.7] at (6.3,0) {$\{+3\}$};
\node[scale = 0.7] at (2.9,-1) {$\{+1\}$};
\node[scale = 0.7] at (6.3,-1) {$\{+2\}$};
\node[scale = 0.7] at (9.7,-1) {$\{+4\}$};
\end{tikzpicture}
\]
Again we only show the sign for the saddle morphisms. In forming $G_2$, the first morphism is again determined by a type A.3, while the dotting morphisms keep their signs because the new saddle leads to a split.

Delooping results in
\[
\begin{tikzpicture}
\torusv{-1}{0}{0.6}{0.6}{very thick}
\torush{2}{0}{0.6}{0.6}{very thick}
\torush{5.4}{0}{0.6}{0.6}{very thick}
\torush{8.8}{0}{0.6}{0.6}{very thick}
\torush{2}{-1}{0.6}{0.6}{very thick}
\torush{5.4}{-1}{0.6}{0.6}{very thick}
\torush{8.8}{-1}{0.6}{0.6}{very thick}
\torush{5.4}{-2}{0.6}{0.6}{very thick}
\draw[->] (-0.2,0.3) -- node [above, scale = 0.7] {$+$} (1.8,0.3);
\draw[->] (-0.2,0.2) -- (1.8,-0.6);
\draw[->] (3.2,0.3) -- (5.2,0.3);
\draw[->] (3.2,0.2) -- (5.2,-0.6);
\draw[->] (3.2,0.1) -- (5.2,-1.6);
\draw[-, line width=6pt, draw=white] (3.2,-0.7) -- (5.2,-0.7);
\draw[->] (3.2,-0.7) -- (5.2,-0.7);
\draw[->] (3.2,-0.8) -- node [above,sloped, scale = 0.7] {$1$} (5.2,-1.7);
\draw[->] (6.6,0.3) -- (8.6,0.3);
\draw[->] (6.6,0.2) -- node [above,sloped,scale = 0.7, near end] {$-1$} (8.6,-0.6);
\draw[-, line width=6pt, draw = white] (6.6,-0.6) -- (8.6, 0.2);
\draw[->] (6.6,-0.6) -- (8.6,0.2);
\draw[->] (6.6,-0.7) -- node [above,scale = 0.7] {$1$} (8.6,-0.7);
\draw[->] (6.6,-1.7) -- (8.6,-0.8);
\torush{7.5}{0.4}{0.2}{0.2}{thick}
\node[scale = 0.5] at (7.61,0.54) {$\bullet$};
\torush{6.8}{-0.4}{0.2}{0.2}{thick}
\node[scale = 0.5] at (6.91, -0.34) {$\bullet$};
\torush{3.85}{0.4}{0.2}{0.2}{thick}
\node[scale = 0.5] at (3.96,0.54) {$\bullet$};
\torush{4.35}{0.4}{0.2}{0.2}{thick}
\node[scale = 0.5] at (4.46, 0.46) {$\bullet$};
\node[scale = 0.7] at (6.65, -0.3) {$-$};
\node[scale = 0.7] at (4.2,0.5) {$-$};
\node[scale = 0.7] at (2.9,0) {$\{+1\}$};
\node[scale = 0.7] at (6.3,0) {$\{+3\}$};
\node[scale = 0.7] at (2.9,-1) {$\{+1\}$};
\node[scale = 0.7] at (6.3,-1) {$\{+3\}$};
\node[scale = 0.7] at (9.7,0) {$\{+5\}$};
\node[scale = 0.7] at (9.7,-1) {$\{+3\}$};
\node[scale = 0.7] at (6.3,-2) {$\{+1\}$};
\end{tikzpicture}
\]
where we do not specify those morphisms that are not going to survive the cancellations. Indeed, we first cancel the lower object in homological degree $1$ with the lowest object of homological degree $2$, and then the middle object of homological degree $2$ with the lower object of homological degree $3$. 

The resulting complex is
\[
E^\ast_3 = 
\begin{tikzpicture}[baseline=1.3ex]
\torusv{-1}{0}{0.6}{0.6}{very thick}
\torush{2}{0}{0.6}{0.6}{very thick}
\torush{5.4}{0}{0.6}{0.6}{very thick}
\torush{8.8}{0}{0.6}{0.6}{very thick}
\draw[->] (-0.2,0.3) -- (1.8,0.3);
\draw[->] (3.2,0.3) -- (5.2,0.3);
\draw[->] (6.6,0.3) -- (8.6,0.3);
\torusv{0.6}{0.4}{0.4}{0.4}{thick}
\draw[->] (0.65,0.6) -- (0.95,0.6);
\torush{3.6}{0.4}{0.4}{0.4}{thick}
\torush{4.4}{0.4}{0.4}{0.4}{thick}
\node[scale = 0.7] at (4.2,0.6) {$-$};
\node[scale = 0.7] at (3.81,0.68) {$\bullet$};
\node[scale = 0.7] at (4.61,0.52) {$\bullet$};
\torush{7}{0.4}{0.4}{0.4}{thick}
\torush{7.8}{0.4}{0.4}{0.4}{thick}
\node[scale = 0.7] at (7.6,0.6) {$-$};
\node[scale = 0.7] at (7.21,0.68) {$\bullet$};
\node[scale = 0.7] at (8.01,0.52) {$\bullet$};
\node[scale = 0.7] at (2.9,0) {$\{+1\}$};
\node[scale = 0.7] at (6.3,0) {$\{+3\}$};
\node[scale = 0.7] at (9.7,0) {$\{+5\}$};
\end{tikzpicture}
\]
At this point we may recall that in even Khovanov homology we would get a $+$ sign between the dotting morphisms from homological degree $2$ to $3$. But in the odd theory we have
\[
\begin{tikzpicture}
\torush{0}{0}{0.6}{0.6}{very thick}
\torush{1}{0}{0.6}{0.6}{very thick}
\node[scale = 0.8] at (0.3,0.42) {$\bullet$};
\node[scale = 0.8] at (1.3,0.18) {$\bullet$};
\draw[dashed] (-0.2,0.8) rectangle (1.8,-0.2);
\draw[dashed] (0.8,0.8) -- (0.8,-0.2);
\node at (2.3,0.3) {$ = -$};
\torush{3}{0}{0.6}{0.6}{very thick}
\torush{4}{0}{0.6}{0.6}{very thick}
\node[scale = 0.8] at (4.3,0.42) {$\bullet$};
\node[scale = 0.8] at (3.3,0.18) {$\bullet$};
\draw[dashed] (2.8,0.8) rectangle (4.8,-0.2);
\draw[dashed] (3.8,0.8) -- (3.8,-0.2);
\end{tikzpicture}
\]
so $E^\ast_3$ is indeed a cochain complex.

Inductively, we see that for $3\leq n < N$ we get
\[
E^\ast_n = 
\begin{tikzpicture}[baseline=1.0ex]
\torusv{0}{0}{0.4}{0.4}{very thick}
\torush{2.4}{0}{0.4}{0.4}{very thick}
\torush{5.2}{0}{0.4}{0.4}{very thick}
\torush{7.6}{0}{0.4}{0.4}{very thick}
\torush{10.4}{0}{0.4}{0.4}{very thick}
\draw[->] (0.6,0.2) -- (2.2,0.2);
\draw[->] (3.4,0.2) -- (5,0.2);
\draw (6.2,0.2) -- (6.5,0.2);
\node at (6.85,0.2) {$\cdots$};
\draw[->] (7.1,0.2) -- (7.4,0.2);
\draw[->] (8.8,0.2) -- (10.2,0.2);
\torusv{1.3}{0.3}{0.2}{0.2}{thick}
\draw[->] (1.33,0.4) -- (1.47,0.4);
\torush{3.85}{0.3}{0.2}{0.2}{thick}
\node[scale = 0.5] at (3.96,0.44) {$\bullet$};
\node[scale = 0.6] at (4.2,0.4) {$-$};
\torush{4.35}{0.3}{0.2}{0.2}{thick}
\node[scale = 0.5] at (4.46, 0.36) {$\bullet$};
\torush{9.15}{0.3}{0.2}{0.2}{thick}
\node[scale = 0.5] at (9.26,0.44) {$\bullet$};
\node[scale = 0.6] at (9.5,0.4) {$-$};
\torush{9.65}{0.3}{0.2}{0.2}{thick}
\node[scale = 0.5] at (9.76, 0.36) {$\bullet$};
\node[scale = 0.6] at (3.1,0) {$\{+1\}$};
\node[scale = 0.6] at (5.9,0) {$\{+3\}$};
\node[scale = 0.6] at (8.45,0) {$\{2n-3\}$};
\node[scale = 0.6] at (11.25,0) {$\{2n-1\}$};
\end{tikzpicture}
\]
To calculate the odd Khovanov homology of the torus link $T(2,n)$, we only need to apply the functor $F_X$ from (\ref{eq:abandon}), with $X$ the usual braid closure, to $E^\ast_n$. After applying $F_X$, the first coboundary is a merge, while all other coboundary maps turn into the zero map. Notice that we still need to do a $q$-grading shift of $n=n_+$ by (\ref{eq:shifted}).

If we choose $X'$ to close the two top ends, and the two bottom ends, the complex $F_{X'}(E^\ast_n)$ is different. In fact, after a $q$-grading shift by $-2n$ and a homological shift by $-n$, this calculates the odd Khovanov homology of the unknot. This requires a little bit more justification than for even Khovanov homology. But notice that if we apply the functors $G_i$ with the closure $X'$ in mind, the first surgery always leads to a type A.2 surgery diagram, and the dotting morphisms still come with a split. So we get the exact same signs as before.
\end{example}

\begin{remark}
There is a reduced version of odd Khovanov homology $\rKh{\ast,\ast}(L)$ such that
\[
\Kho{i,j}(L)\cong \rKh{i,j-1}(L)\oplus \rKh{i,j+1}(L),
\]
see \cite[Prop.1.8]{MR3071132}. This can be calculated by choosing a basepoint $p$ on the link diagram, and then considering the subcomplex $a_p\wedge CO^\ast(L)\{+1\} \subset CO^\ast(L)\{+1\}$. Here $a_p$ refers to the generator corresponding to the component with the basepoint.

We can calculate the reduced version by running our algorithm in the same way for the first $n-1$ crossings, then place the basepoint near the last crossing, and form the subcomplex after the delooping.
\end{remark}

The algorithm has been implemented in the computer program \verb+KnotJob+ available from the author's website. Compared to a similar implementation for even Khovanov homology the odd version is slightly slower and less memory efficient, which is unsurprising given that one has to be more careful with signs. This does not seem to affect the complexity though, and we expect that if the even homology of a knot can be calculated, then the odd homology can also be calculated.

As a sample calculation, we list the reduced odd Khovanov homology of the $(8,9)$-torus knot in Figures \ref{fig:torus892} and \ref{fig:torus893}. The calculation took about half an hour on a standard PC with 4GB RAM.

\begin{figure}[ht]
\begin{center}
\includegraphics[width = 11.5cm]{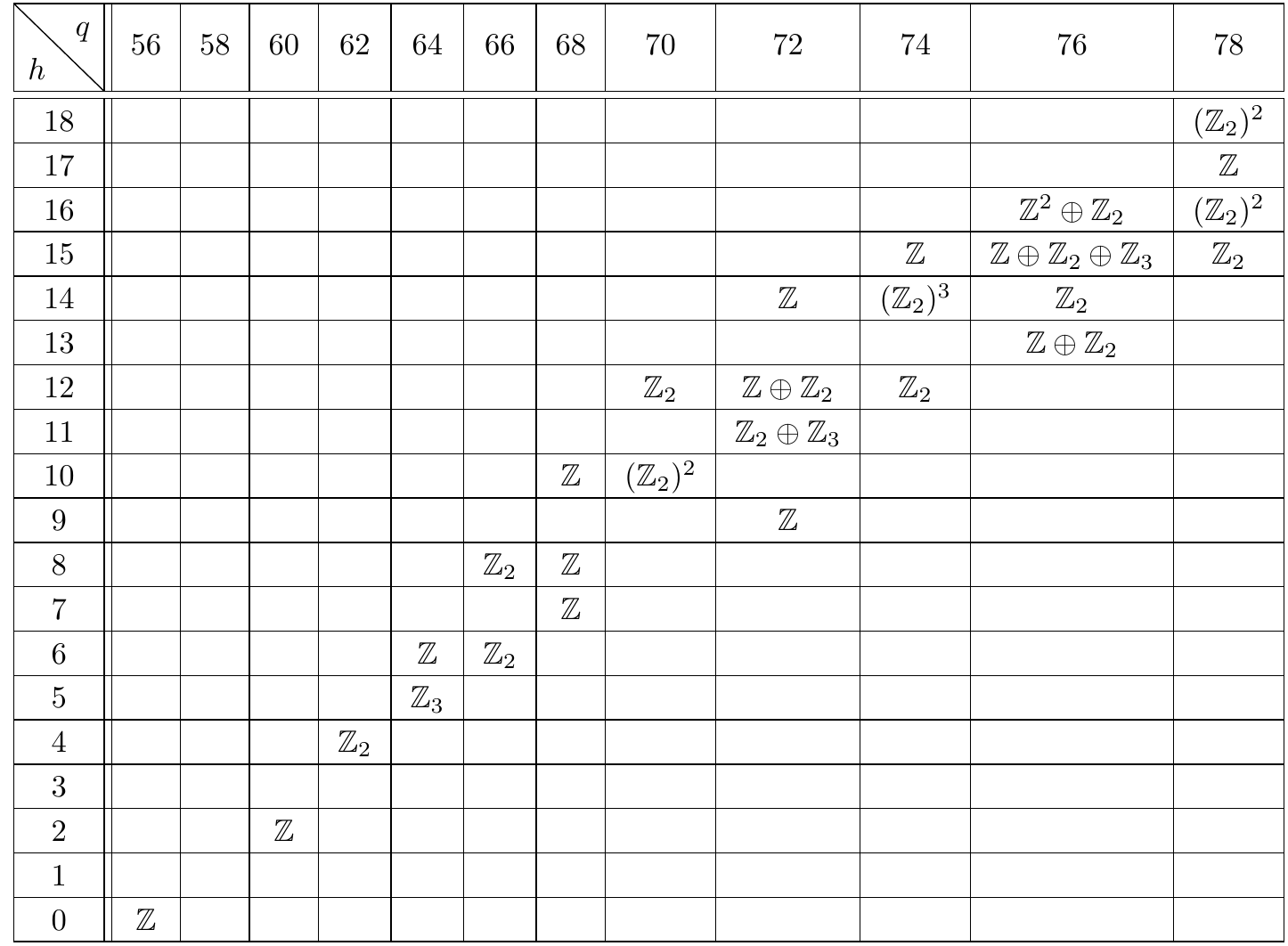}

\vspace{0.5cm}

\includegraphics[width = 11.5cm]{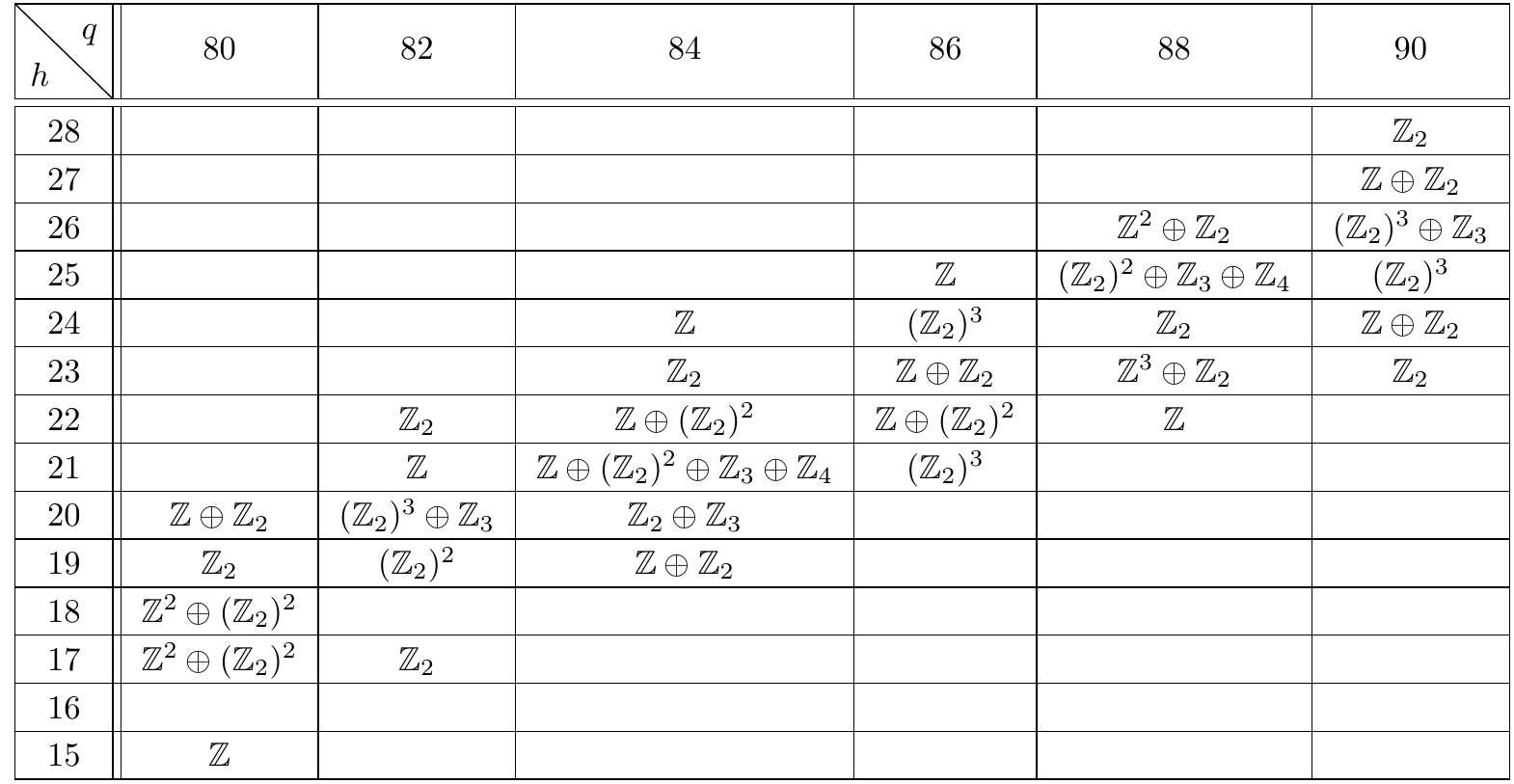}
\caption{\label{fig:torus892} The reduced odd Khovanov homology of the $(8,9)$-torus knot in quantum degrees $56$ to $90$.}
\end{center}
\end{figure}

\begin{figure}[ht]
\begin{center}
\includegraphics[width = 11.5cm]{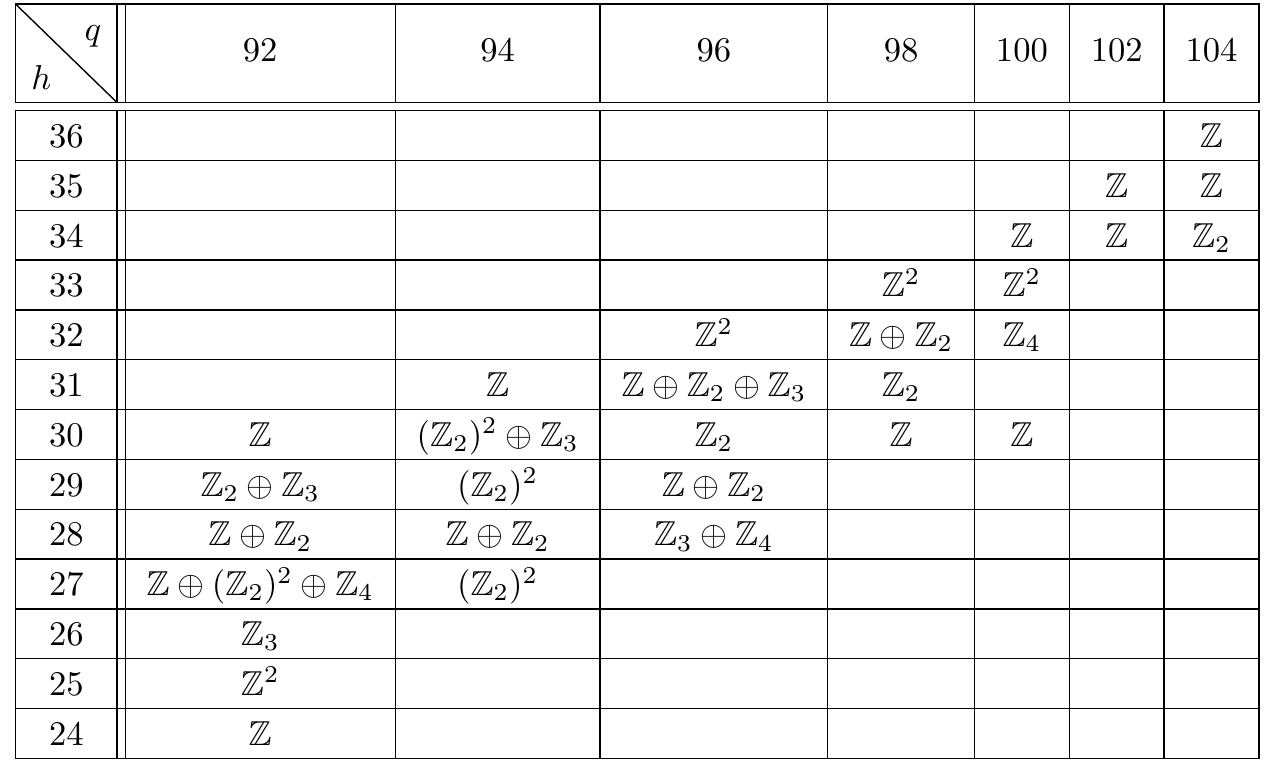}
\caption{\label{fig:torus893} The reduced odd Khovanov homology of the $(8,9)$-torus knot in quantum degrees $92$ to $104$.}
\end{center}
\end{figure}

Our program does not put a lot of effort into comparing chronological cobordisms. However, the following lemma, which describes a situation that arises frequently and can be used to simplify chronologies, has been implemented.

\begin{lemma}\label{lm:dotsforsurg}
In $\Z\dChrqu{X}(B,\dot{B})$ we have
\begin{equation}\label{eq:doublesurg}
\begin{tikzpicture}[baseline={([yshift=-.5ex]current bounding box.center)}]
\smoothingup{0}{0}{0.6}{very thick}
\smoothinglr{1}{0}{0.6}{very thick}
\draw (0.2,0.3) -- (0.4,0.3);
\draw[->] (1.3,0.2) -- (1.3, 0.4);
\draw[dashed] (-0.2,0.8) rectangle (1.8,-0.2);
\draw[dashed] (0.8,0.8) -- (0.8, -0.2);
\node at (3,0.3) {$= (-1)^{1+\varphi(T)}$};
\node[scale = 2.5] at (4.2,0.3) {$($};
\smoothingup{4.6}{0}{0.6}{very thick}
\node[scale = 0.8] at (4.8,0.3) {$\bullet$};
\draw[dashed] (4.4,0.8) rectangle (5.4,-0.2);
\node at (5.7,0.3) {$-$};
\smoothingup{6.2}{0}{0.6}{very thick}
\node[scale = 0.8] at (6.615,0.3) {$\bullet$};
\draw[dashed] (6,0.8) rectangle (7,-0.2);
\node[scale = 2.5] at (7.2,0.3) {$)$};
\end{tikzpicture}
\end{equation}
where $T\in\mathcal{S}$ represents a surgery diagram corresponding to the two surgeries on the left-hand side, after closing with $X$.
\end{lemma}

Notice that the orientation of the first surgery is not shown, but it can have an effect on $\varphi(T)$. In fact, $T$ is either of type X, Y, or C.2. The second orientation is shown, as it determines the order in which we write the two dottings on the right-hand side.

\begin{proof}
We can write $T$ locally as
\[
\begin{tikzpicture}
\smoothingup{0}{0}{1}{very thick}
\draw (0.333,0.5) -- (0.667,0.5);
\draw[->] (0.25, 0.25) to [out = 135, in = 225] (0.25, 0.75);
\end{tikzpicture}
\]
and therefore
\[
\begin{tikzpicture}
\smoothingup{0}{0}{0.6}{very thick}
\smoothinglr{1}{0}{0.6}{very thick}
\draw (0.2,0.3) -- (0.4,0.3);
\draw[->] (1.3,0.2) -- (1.3, 0.4);
\draw[dashed] (-0.2,0.8) rectangle (1.8,-0.2);
\draw[dashed] (0.8,0.8) -- (0.8, -0.2);
\node at (3,0.3) {$= (-1)^{1+\varphi(T)}$};
\node[scale = 2.5] at (4.2,0.3) {$($};
\smoothingup{4.6}{0}{0.6}{very thick}
\draw[->] (4.73, 0.15) to [out = 135, in = 225] (4.73, 0.45);
\widesmoothingup{5.6}{0}{0.6}{very thick}
\draw[very thick] (6.2,0.3) circle (0.2);
\draw (6.4,0.3) -- (6.6,0.3);
\draw[dashed] (4.4,0.8) rectangle (7,-0.2);
\draw[dashed] (5.4,0.8) -- (5.4,-0.2);
\node[scale = 2.5] at (7.2,0.3) {$)$};
\end{tikzpicture}
\]
From the neck-cutting relation we get
\[
\begin{tikzpicture}
\smoothingup{0}{0}{0.6}{very thick}
\draw[->] (0.13, 0.15) to [out = 135, in = 225] (0.13, 0.45);
\widesmoothingup{1}{0}{0.6}{very thick}
\draw[very thick] (1.6,0.3) circle (0.2);
\draw (1.8,0.3) -- (2,0.3);
\draw[dashed] (-0.2,0.8) rectangle (2.4,-0.2);
\draw[dashed] (0.8,0.8) -- (0.8,-0.2);
\node at (2.7,0.3) {$=$};
\smoothingup{3.2}{0}{0.6}{very thick}
\draw[->] (3.33, 0.15) to [out = 135, in = 225] (3.33, 0.45);
\widesmoothingup{4.2}{0}{0.6}{very thick}
\death{4.8}{0.3}{0.2}{very thick}
\widesmoothingup{5.8}{0}{0.6}{very thick}
\birth{6.4}{0.3}{0.2}{very thick}
\widesmoothingup{7.4}{0}{0.6}{very thick}
\draw[very thick] (8,0.3) circle (0.2);
\node[scale = 0.8] at (8,0.1) {$\bullet$}; 
\widesmoothingup{9}{0}{0.6}{very thick}
\draw[very thick] (9.6,0.3) circle (0.2);
\draw (9.8,0.3) -- (10,0.3);
\draw[dashed] (3,0.8) rectangle (10.4,-0.2);
\draw[dashed] (4,0.8) -- (4, -0.2);
\draw[dashed] (5.6, 0.8) -- (5.6,-0.2);
\draw[dashed] (7.2, 0.8) -- (7.2,-0.2);
\draw[dashed] (8.8, 0.8) -- (8.8,-0.2);
\node at (3, -0.9) {$+$};
\smoothingup{3.5}{-1.2}{0.6}{very thick}
\draw[->] (3.63, -1.05) to [out = 135, in = 225] (3.63, -0.75);
\widesmoothingup{4.5}{-1.2}{0.6}{very thick}
\death{6.7}{-0.9}{0.2}{very thick}
\widesmoothingup{6.1}{-1.2}{0.6}{very thick}
\birth{8.3}{-0.9}{0.2}{very thick}
\widesmoothingup{7.7}{-1.2}{0.6}{very thick}
\draw[very thick] (5.1,-0.9) circle (0.2);
\node[scale = 0.8] at (5.1,-1.1) {$\bullet$}; 
\widesmoothingup{9.3}{-1.2}{0.6}{very thick}
\draw[very thick] (9.9,-0.9) circle (0.2);
\draw (10.1,-0.9) -- (10.3,-0.9);
\draw[dashed] (3.3,-0.4) rectangle (10.7,-1.4);
\draw[dashed] (4.3,-0.4) -- (4.3, -1.4);
\draw[dashed] (5.9, -0.4) -- (5.9,-1.4);
\draw[dashed] (7.5, -0.4) -- (7.5,-1.4);
\draw[dashed] (9.1, -0.4) -- (9.1,-1.4);
\node at (2.7, -2.1) {$=$};
\node at (3.05,-2.1) {$-$};
\smoothingup{3.5}{-2.4}{0.6}{very thick}
\node[scale = 0.8] at (3.915,-2.1) {$\bullet$};
\draw[dashed] (3.3,-1.6) rectangle (4.3,-2.6);
\node at (4.6,-2.1) {$+$};
\smoothingup{5.1}{-2.4}{0.6}{very thick}
\node[scale = 0.8] at (5.3,-2.1) {$\bullet$};
\draw[dashed] (4.9,-1.6) rectangle (5.9,-2.6);
\end{tikzpicture}
\]
which implies the result.
\end{proof}

\section{Three strand torus links}

Let $B\subset \R^2$ be a disc and $\dot{B}\subset \partial B$ consist of $6$ points, three in the lower half and three in the upper half of $\partial B$. Also let $X\subset \R^2-(B-\partial B)$ be the braid closure of the points in $\dot{B}$. Consider the following infinitely generated cochain complex $C^\ast$ over $\Z\dChrqu{X}(B,\dot{B})$, which continues with period $8$.
\[
\begin{tikzpicture}
\smomega{0}{0}{0.8}{0.8}{very thick}
\draw[->] (1,0.4) -- node [above] {$A$} (1.7, 0.4);
\node[scale = 2] at (1.9, 0.4) {$($};
\smgamma{2.2}{0}{0.8}{0.8}{very thick}
\node at (3.3,0.4) {$\oplus$};
\smdelta{3.6}{0}{0.8}{0.8}{very thick}
\node[scale = 2] at (4.7,0.4) {$)$};
\node[scale = 0.7] at (4.9,0.1) {$\{1\}$};
\draw[->] (5, 0.4) -- node [above] {$B_1$} (5.7, 0.4);
\node[scale = 2] at (5.9, 0.4) {$($};
\smbeta{6.2}{0}{0.8}{0.8}{very thick}
\node at (7.3,0.4) {$\oplus$};
\smalpha{7.6}{0}{0.8}{0.8}{very thick}
\node[scale = 2] at (8.7,0.4) {$)$};
\node[scale = 0.7] at (8.9,0.1) {$\{2\}$};
\draw[->] (8.9,0.4) -- (9.1, 0.4) to [out = 0, in = 0] (9.1, -0.2) -- (-1.7, -0.2) to [out = 180, in = 180] (-1.7, -0.8) -- (-1.6, -0.8);
\node at (9.6,0.1) {$C_1$};
\node[scale = 2] at (-1.4, -0.8) {$($};
\smbeta{-1.1}{-1.2}{0.8}{0.8}{very thick}
\node at (0, -0.8) {$\oplus$};
\smalpha{0.3}{-1.2}{0.8}{0.8}{very thick}
\node[scale = 2] at (1.4,-0.8) {$)$};
\node[scale = 0.7] at (1.6, -1.1) {$\{4\}$};
\draw[->] (1.7, -0.8) -- node [above] {$D_1$} (2.4, -0.8);
\node[scale = 2] at (2.6, -0.8) {$($};
\smgamma{2.9}{-1.2}{0.8}{0.8}{very thick}
\node at (4,-0.8) {$\oplus$};
\smdelta{4.3}{-1.2}{0.8}{0.8}{very thick}
\node[scale = 2] at (5.4,-0.8) {$)$};
\node[scale = 0.7] at (5.6,-1.1) {$\{5\}$};
\draw[->] (5.7,-0.8) -- node [above] {$E_1$} (6.4, -0.8);
\node[scale = 2] at (6.6, -0.8) {$($};
\smgamma{6.9}{-1.2}{0.8}{0.8}{very thick}
\node at (8,-0.8) {$\oplus$};
\smdelta{8.3}{-1.2}{0.8}{0.8}{very thick}
\node[scale = 2] at (9.4,-0.8) {$)$};
\node[scale = 0.7] at (9.6,-1.1) {$\{7\}$};
\draw[->] (9.6, -0.8) -- (9.8, -0.8) to [out = 0, in = 0] (9.8, -1.4) -- (-1.7, -1.4) to [out = 180, in = 180] (-1.7, -2) -- (-1.6, -2);
\node at (10.3, -1.1) {$B_2$};
\node[scale = 2] at (-1.4, -2) {$($};
\smbeta{-1.1}{-2.4}{0.8}{0.8}{very thick}
\node at (0, -2) {$\oplus$};
\smalpha{0.3}{-2.4}{0.8}{0.8}{very thick}
\node[scale = 2] at (1.4,-2) {$)$};
\node[scale = 0.7] at (1.6, -2.3) {$\{8\}$};
\draw[->] (1.7, -2) -- node [above] {$C_2$} (2.4, -2);
\node[scale = 2] at (2.6, -2) {$($};
\smbeta{2.9}{-2.4}{0.8}{0.8}{very thick}
\node at (4,-2) {$\oplus$};
\smalpha{4.3}{-2.4}{0.8}{0.8}{very thick}
\node[scale = 2] at (5.4,-2) {$)$};
\node[scale = 0.7] at (5.65,-2.3) {$\{10\}$};
\draw[->] (5.7,-2) -- node [above] {$D_2$} (6.4, -2);
\node[scale = 2] at (6.6, -2) {$($};
\smgamma{6.9}{-2.4}{0.8}{0.8}{very thick}
\node at (8,-2) {$\oplus$};
\smdelta{8.3}{-2.4}{0.8}{0.8}{very thick}
\node[scale = 2] at (9.4,-2) {$)$};
\node[scale = 0.7] at (9.65,-2.3) {$\{11\}$};
\draw[->] (9.6, -2) -- (9.8, -2) to [out = 0, in = 0] (9.8, -2.6) -- (-1.7, -2.6) to [out = 180, in = 180] (-1.7, -3.2) -- (-1.6, -3.2);
\node at (10.3, -2.3) {$E_2$};
\node[scale = 2] at (-1.4, -3.2) {$($};
\smgamma{-1.1}{-3.6}{0.8}{0.8}{very thick}
\node at (0, -3.2) {$\oplus$};
\smdelta{0.3}{-3.6}{0.8}{0.8}{very thick}
\node[scale = 2] at (1.4,-3.2) {$)$};
\node[scale = 0.7] at (1.65, -3.5) {$\{13\}$};
\draw[->] (1.7, -3.2) -- node [above] {$B_1$} (2.4, -3.2);
\node[scale = 2] at (2.6, -3.2) {$($};
\smbeta{2.9}{-3.6}{0.8}{0.8}{very thick}
\node at (4,-3.2) {$\oplus$};
\smalpha{4.3}{-3.6}{0.8}{0.8}{very thick}
\node[scale = 2] at (5.4,-3.2) {$)$};
\node[scale = 0.7] at (5.65,-3.5) {$\{14\}$};
\draw[->] (5.7,-3.2) -- node [above] {$C_1$} (6.4, -3.2);
\node at (6.8, -3.2) {$\cdots$};
\end{tikzpicture}
\]
where
\[
A = \left( \, \, \, 
\begin{tikzpicture}[baseline=4ex]
\smomega{0.8}{0.8}{0.6}{0.6}{very thick}
\draw[->] (0.8, 1.1) -- (1.1, 1.1);
\smomega{0.8}{0}{0.6}{0.6}{very thick}
\draw[->] (1.1, 0.3) -- (1.4, 0.3);
\end{tikzpicture}
\, \, \, \right), \hspace{0.3cm}
B_1 = \left( 
\begin{tikzpicture}[baseline=4ex]
\smgamma{1}{0.8}{0.6}{0.6}{very thick}
\draw[->] (1.3, 1.3) -- (1.6, 1.3);
\node at (0.8, 0.3) {$-$};
\smgamma{1}{0}{0.6}{0.6}{very thick}
\draw[->] (1.3, 0.1) -- (1.6, 0.1);
\node at (2.2, 1.1) {$-$};
\smdelta{2.4}{0.8}{0.6}{0.6}{very thick}
\draw[->] (2.4,0.9) -- (2.7, 0.9);
\smdelta{2.4}{0}{0.6}{0.6}{very thick}
\draw[->] (2.4, 0.5) -- (2.7, 0.5);
\end{tikzpicture}
\, \, \, \right), \hspace{0.3cm}
B_2 = \left( 
\begin{tikzpicture}[baseline=4ex]
\smgamma{1}{0.8}{0.6}{0.6}{very thick}
\draw[->] (1.3, 1.3) -- (1.6, 1.3);
\node at (0.8, 0.3) {$-$};
\smgamma{1}{0}{0.6}{0.6}{very thick}
\draw[->] (1.3, 0.1) -- (1.6, 0.1);
\smdelta{2.4}{0.8}{0.6}{0.6}{very thick}
\draw[->] (2.4, 0.9) -- (2.7, 0.9);
\smdelta{2.4}{0}{0.6}{0.6}{very thick}
\draw[->] (2.4, 0.5) -- (2.7, 0.5);
\end{tikzpicture}
\, \, \, \right),
\]
\[
C_1 = \left( \, \, \, 
\begin{tikzpicture}[baseline=4ex]
\dottedcup{1}{0.8}{0.6}{0.6}{very thick}
\node at (1.6,1.1) {$-2$};
\dottedbslash{1.9}{0.8}{0.3}{0.6}{very thick}
\node at (2.4,1.1) {$+$};
\dottedcap{2.6}{0.8}{0.6}{0.6}{very thick}
\smbeta{1.2}{0}{0.6}{0.6}{very thick}
\draw[->] (1.2,0.5) -- (1.5, 0.5);
\draw[dashed] (2, 0.7) -- (2, -0.1);
\smgamma{2.2}{0}{0.6}{0.6}{very thick}
\draw[->] (2.5, 0.1) -- (2.8, 0.1);
\smalpha{3.6}{0.8}{0.6}{0.6}{very thick}
\draw[->] (3.9,1.3) -- (4.2, 1.3);
\draw[dashed] (4.4,1.5) -- (4.4, 0.7);
\smdelta{4.6}{0.8}{0.6}{0.6}{very thick}
\draw[->] (4.6, 0.9) -- (4.9, 0.9);
\dottedcup{3.9}{0}{0.6}{0.6}{very thick}
\node at (4.4, 0.3) {$-$};
\dottedcap{4.6}{0}{0.6}{0.6}{very thick}
\end{tikzpicture}
\, \, \, \right), \, \,
C_2 = \left( \, \, \, 
\begin{tikzpicture}[baseline=4ex]
\dottedcup{1.5}{0.8}{0.6}{0.6}{very thick}
\node at (2,1.1) {$-$};
\dottedcap{2.2}{0.8}{0.6}{0.6}{very thick}
\smbeta{1.2}{0}{0.6}{0.6}{very thick}
\draw[->] (1.2,0.5) -- (1.5, 0.5);
\draw[dashed] (2, 0.7) -- (2, -0.1);
\smgamma{2.2}{0}{0.6}{0.6}{very thick}
\draw[->] (2.5, 0.1) -- (2.8, 0.1);
\node at (3.4, 1.1) {$-$};
\smalpha{3.6}{0.8}{0.6}{0.6}{very thick}
\draw[->] (3.9,1.3) -- (4.2, 1.3);
\draw[dashed] (4.4,1.5) -- (4.4, 0.7);
\smdelta{4.6}{0.8}{0.6}{0.6}{very thick}
\draw[->] (4.6, 0.9) -- (4.9, 0.9);
\dottedcup{3.4}{0}{0.6}{0.6}{very thick}
\node at (4,0.3) {$-2$};
\dottedslash{4.3}{0}{0.3}{0.6}{very thick}
\node at (4.8,0.3) {$+$};
\dottedcap{5}{0}{0.6}{0.6}{very thick}
\end{tikzpicture}
\, \, \, \right),
\]
\[
D_1 = \left( \, \, \, 
\begin{tikzpicture}[baseline=4ex]
\smbeta{1}{0.8}{0.6}{0.6}{very thick}
\draw[->] (1, 1.3) -- (1.3, 1.3);
\smbeta{1}{0}{0.6}{0.6}{very thick}
\draw[->] (1.3, 0.1) -- (1.6, 0.1);
\node at (2.2, 1.1) {$-$};
\smalpha{2.4}{0.8}{0.6}{0.6}{very thick}
\draw[->] (2.4,0.9) -- (2.7, 0.9);
\smalpha{2.4}{0}{0.6}{0.6}{very thick}
\draw[->] (2.7, 0.5) -- (3, 0.5);
\end{tikzpicture}
\, \, \, \right), \hspace{0.3cm}
D_2 = \left( \, \, \, 
\begin{tikzpicture}[baseline=4ex]
\smbeta{1}{0.8}{0.6}{0.6}{very thick}
\draw[->] (1, 1.3) -- (1.3, 1.3);
\smbeta{1}{0}{0.6}{0.6}{very thick}
\draw[->] (1.3, 0.1) -- (1.6, 0.1);
\smalpha{2.4}{0.8}{0.6}{0.6}{very thick}
\draw[->] (2.4, 0.9) -- (2.7, 0.9);
\smalpha{2.4}{0}{0.6}{0.6}{very thick}
\draw[->] (2.7, 0.5) -- (3, 0.5);
\end{tikzpicture}
\, \, \, \right),
\]
and
\[
E_1 = \left( \, \, \, 
\begin{tikzpicture}[baseline=4ex]
\dottedcup{1.2}{0.8}{0.6}{0.6}{very thick}
\node at (1.8,1.1) {$-2$};
\dottedline{2.15}{0.8}{0.6}{very thick}
\node at (2.4,1.1) {$+$};
\dottedcap{2.55}{0.8}{0.6}{0.6}{very thick}
\smgamma{1.2}{0}{0.6}{0.6}{very thick}
\draw[->] (1.5,0.5) -- (1.8, 0.5);
\draw[dashed] (2, 0.7) -- (2, -0.1);
\smbeta{2.2}{0}{0.6}{0.6}{very thick}
\draw[->] (2.5, 0.1) -- (2.8, 0.1);
\smdelta{3.6}{0.8}{0.6}{0.6}{very thick}
\draw[->] (3.6,1.3) -- (3.9, 1.3);
\draw[dashed] (4.4,1.5) -- (4.4, 0.7);
\smalpha{4.6}{0.8}{0.6}{0.6}{very thick}
\draw[->] (4.6, 0.9) -- (4.9, 0.9);
\dottedcup{3.6}{0}{0.6}{0.6}{very thick}
\node at (4.2, 0.3) {$-2$};
\dottedline{4.55}{0}{0.6}{very thick}
\node at (4.8,0.3) {$+$};
\dottedcap{4.95}{0}{0.6}{0.6}{very thick}
\end{tikzpicture}
\, \, \, \right), \, \,
E_2 = \left( \, \, \, 
\begin{tikzpicture}[baseline=4ex]
\dottedcup{1.5}{0.8}{0.6}{0.6}{very thick}
\node at (2,1.1) {$-$};
\dottedcap{2.2}{0.8}{0.6}{0.6}{very thick}
\smgamma{1.2}{0}{0.6}{0.6}{very thick}
\draw[->] (1.5,0.5) -- (1.8, 0.5);
\draw[dashed] (2, 0.7) -- (2, -0.1);
\smbeta{2.2}{0}{0.6}{0.6}{very thick}
\draw[->] (2.5, 0.1) -- (2.8, 0.1);
\node at (3.4, 1.1) {$-$};
\smdelta{3.6}{0.8}{0.6}{0.6}{very thick}
\draw[->] (3.6,1.3) -- (3.9, 1.3);
\draw[dashed] (4.4,1.5) -- (4.4, 0.7);
\smalpha{4.6}{0.8}{0.6}{0.6}{very thick}
\draw[->] (4.6, 0.9) -- (4.9, 0.9);
\dottedcup{3.9}{0}{0.6}{0.6}{very thick}
\node at (4.4,0.3) {$-$};
\dottedcap{4.6}{0}{0.6}{0.6}{very thick}
\end{tikzpicture}
\, \, \, \right).
\]
\begin{remark}
One should compare this complex with the categorification of the third Jones-Wenzl projector given in \cite[\S 4.4]{MR2901969}. A striking difference is that $C^\ast$ has period $8$, while the complex in \cite{MR2901969} has period $4$. The subtle sign differences described in the coboundary above do have an effect, as the homology also has period $8$, and not, as in the even Khovanov homology situation of \cite{MR2901969}, period $4$.
\end{remark}

We claim that the odd Khovanov homology of the $(3,n)$-torus link $T(3,n)$ can be calculated by forming an appropriate quotient complex of $C^\ast$ and applying the functors $F_X$ from (\ref{eq:abandon}) and $F$ from Lemma \ref{lm:functor}. For organisatorial purposes, let us write $C_k^\ast$ for the subcomplex of $C^\ast$ containing all generators of homological degree $>k$, and $C_{k-\frac{1}{2}}^\ast$ for the subcomplex of $C^\ast$ containing all generators of homological degree $>k$, 
and also the second generator of homological degree $k$ (in the order suggested by the direct sum symbol in the definition of $C^\ast$). Here we assume $k\geq 1$ and that $C^\ast$ starts in degree $0$. We then write
\[
KO^\ast_k = F(F_X(C^\ast / C^\ast_k)) \, \, \, \mbox{ and } \, \, \, KO^\ast_{k-\frac{1}{2}} = F(F_X(C^\ast / C^\ast_{k-\frac{1}{2}}))
\]
\begin{theorem}\label{thm:3strand}
Let $n\geq 0$. Then
\begin{align*}
CO^\ast(T(3,6n+1))  & \simeq KO^\ast_{2+8n-\frac{1}{2}}\{12n+2\} \\
CO^\ast(T(3,6n+2))  & \simeq KO^\ast_{3+8n-\frac{1}{2}}\{12n+4\} \\
CO^\ast(T(3,6n+3))  & \simeq KO^\ast_{5+8n}\{12n+6\} \\
CO^\ast(T(3,6n+4))  & \simeq KO^\ast_{6+8n-\frac{1}{2}}\{12n+8\} \\
CO^\ast(T(3,6n+5))  & \simeq KO^\ast_{7+8n-\frac{1}{2}}\{12n+10\} \\
CO^\ast(T(3,6n+6))  & \simeq KO^\ast_{9+8n}\{12(n+1)\}.
\end{align*}
Here $\simeq$ stands for chain homotopy equivalent as $q$-graded cochain complexes.
\end{theorem}

The proof is a lengthy induction, and we will only give the basic idea. Essentially we simply run our algorithm on the standard braid diagram for a $3$-strand torus link. The two braid letters give the functors $F_1, G_1\colon \Z\dChrqu{X}(B,\dot{B}) \to \Z\dChrqu{X}(B,\dot{B})$ and $F_2, G_2\colon \Z\dChrqu{X}(B,\dot{B}) \to \Z\dChrqu{X}(B,\dot{B})$, respectively.

To get from the statement for $T(3,6n+i)$ to the statement for $T(3, 6n+i+1)$, one has to form the mapping cone for the natural transformation between $F_1$ and $G_1$ of the appropriate quotient of $C^\ast$, deloop and Gauss eliminate, then repeat with the mapping cone for the natural transformation between $F_2$ and $G_2$.

It is worth pointing out that the smoothings do not change when the functors $F_1$ or $F_2$ are applied. However, applying the functors $G_1$ or $G_2$ on generators will create a circle in the smoothings of exactly one generator in every homological degree $h\geq 1$. After delooping we can cancel the new generators that arose from $G_1$ or $G_2$, by starting in homological degree $1$ and working our way up.

The generators that arose from $F_1$ or $F_2$ remain, while all generators that arose from $G_1$ or $G_2$ get cancelled, except for one or two at the highest homological degree. To get the dottings on the diagonals of the $C$ and $E$ matrices we also need to apply Lemma \ref{lm:dotsforsurg} after some cancellations.

In fact, it is easy to see this way that $C^\ast\simeq \mathcal{C}^{\sigma_i}_{F_i,G_i}(C)^\ast$ for $i=1,2$, where $\sigma_i$ is the natural transformation between $F_i$ and $G_i$. We used our computer program to determine $C^\ast$, but it is instructive to check at least one or two of the inductive steps by hand to see how the complex arises. The same technique also works in even Khovanov homology, and there one has noticeably less work because of the smaller period.

It is straightforward to calculate the odd Khovanov homology of any $3$-strand torus link from Theorem \ref{thm:3strand}, but even with periodicity there are still quite a few cases. As an Example, we list the homology of $T(3,14)$ in Figure \ref{fig:torus314}. The periodic part is nicely visible, and the top homology behaves the same way for all torus knots $T(3,6n+2)$, suitably shifted.

\begin{figure}[ht]
\begin{center}
\includegraphics[width = 12cm]{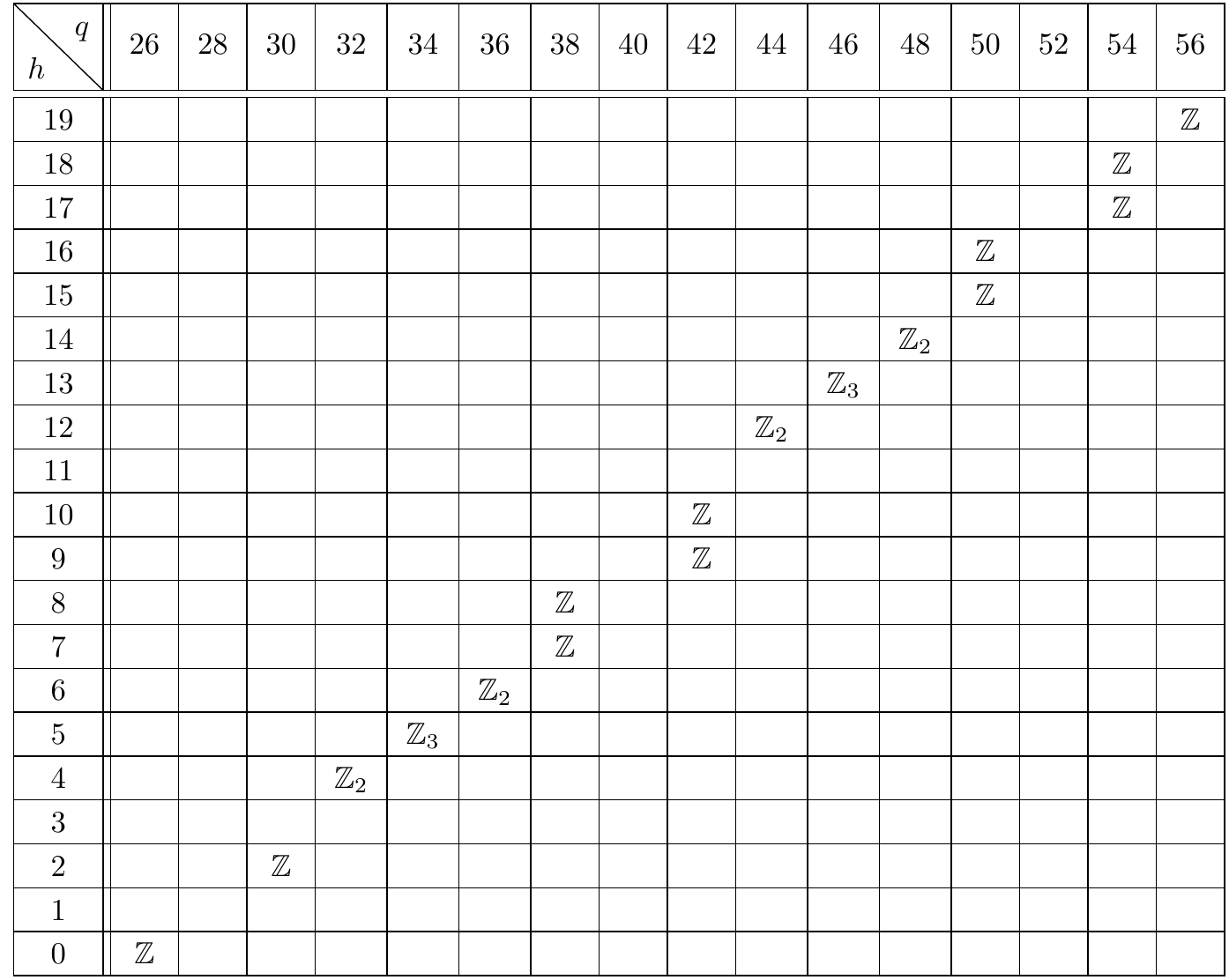}
\end{center}
\caption{\label{fig:torus314} The reduced odd Khovanov homology of $T(3,14)$.}
\end{figure} 

As a sample calculation, let us show how the $3$-torsion in $q$-degree $34$ arises from $C^\ast$.

\begin{example}
Since we have $28$ positive crossings in $T(3,14)$, we consider $C^\ast\{28\}$. We focus on homological degrees $4$ to $6$ of this complex to get
\[
\begin{tikzpicture}
\smgamma{0}{2}{0.8}{0.8}{very thick}
\closurex{0}{2}{0.8}{0.8}{very thick}
\node at (0.4,1.4) {$\oplus$};
\smdelta{0}{0}{0.8}{0.8}{very thick}
\closurex{0}{0}{0.8}{0.8}{very thick}
\node[scale = 0.8] at (1.4, 1.7) {$\{33\}$};
\node[scale = 0.8] at (1.4, -0.3) {$\{33\}$};
\draw[->] (1.5, 2.4) -- (3.5, 2.4);
\node[scale = 0.7] at (2,2.7) {$2\,($};
\dottedcup{2.2}{2.5}{0.4}{0.4}{thick}
\node[scale = 0.7] at (2.6,2.7) {$-$};
\dottedline{2.8}{2.5}{0.4}{thick}
\node[scale = 0.7] at (3,2.7) {$)$};
\draw[->] (1.5, 2.2) -- node [above, sloped, near end, scale = 0.8] {$A.3$} (3.5, 0.6);
\draw[-, line width=6pt, color=white] (1.5, 0.6) -- (3.5, 2.2);
\draw[->] (1.5, 0.6) -- node [above, sloped, near end, scale = 0.8] {$A.3$} (3.5, 2.2);
\draw[->] (1.5, 0.4) -- (3.5, 0.4);
\node[scale = 0.7] at (2,0.7) {$2\,($};
\dottedcup{2.2}{0.5}{0.4}{0.4}{thick}
\node[scale = 0.7] at (2.6,0.7) {$-$};
\dottedline{2.8}{0.5}{0.4}{thick}
\node[scale = 0.7] at (3,0.7) {$)$};
\smgamma{4}{2}{0.8}{0.8}{very thick}
\closurex{4}{2}{0.8}{0.8}{very thick}
\node at (4.4,1.4) {$\oplus$};
\smdelta{4}{0}{0.8}{0.8}{very thick}
\closurex{4}{0}{0.8}{0.8}{very thick}
\node[scale = 0.8] at (5.4, 1.7) {$\{35\}$};
\node[scale = 0.8] at (5.4, -0.3) {$\{35\}$};
\draw[->] (5.5, 2.4) -- node [above, scale = 0.8] {$M$} (7.5,2.4);
\draw[->] (5.5, 2.2) -- node [above, sloped, near end, scale = 0.8] {$-M$} (7.5, 0.6);
\draw[-, line width=6pt, color=white] (5.5, 0.6) -- (7.5, 2.2);
\draw[->] (5.5, 0.6) -- node [above, sloped, near end, scale = 0.8] {$M$} (7.5, 2.2);
\draw[->] (5.5, 0.4) -- node [above, scale = 0.8] {$M$} (7.5, 0.4);
\smbeta{8}{2}{0.8}{0.8}{very thick}
\closurex{8}{2}{0.8}{0.8}{very thick}
\node at (8.4,1.4) {$\oplus$};
\smalpha{8}{0}{0.8}{0.8}{very thick}
\closurex{8}{0}{0.8}{0.8}{very thick}
\node[scale = 0.8] at (9.4, 1.7) {$\{36\}$};
\node[scale = 0.8] at (9.4, -0.3) {$\{36\}$};
\end{tikzpicture}
\]
after applying $F_X$. Here the first set of morphisms is governed by $E_1$ and the second by $B_2$. Also, $M$ stands for a merge cobordism, and $A.3$ for two surgeries whose surgery diagram is of type A.3. Since the order matters for those, we note that the higher placed surgery arc is resolved first.


Applying $F$ gives
\[
\begin{tikzpicture}
\node at (0,0) {$\Lambda^\ast(a,b)\{33\}$};
\node at (0,0.65) {$\oplus$};
\node at (0, 1.3) {$\Lambda^\ast(a,b)\{33\}$};
\draw[->] (1.1, 1.2) -- node [above, sloped, scale = 0.7, near end] {$a-b$} (2.9,0.1);
\draw[-, line width=6pt, color = white] (1.1, 0.1) -- (2.9, 1.2);
\draw[->] (1.1, 1.3) -- node [above, scale = 0.7] {$2(a-b)$} (2.9, 1.3);
\draw[->] (1.1, 0.1) -- node [above, sloped, scale = 0.7, near end] {$b-a$} (2.9, 1.2);
\draw[->] (1.1, 0) -- node [above, scale = 0.7] {$2(b-a)$} (2.9, 0);
\node at (4,0) {$\Lambda^\ast(a,b)\{35\}$};
\node at (4,0.65) {$\oplus$};
\node at (4, 1.3) {$\Lambda^\ast(a,b)\{35\}$};
\node at (8,0) {$\Lambda^\ast(a)\{36\}$};
\node at (8,0.65) {$\oplus$};
\node at (8, 1.3) {$\Lambda^\ast(a)\{36\}$};
\draw[->] (5.1, 1.2) -- node [above, sloped, scale = 0.7, near end] {$-i$} (7,0.1);
\draw[-, line width=6pt, color = white] (5.1, 0.1) -- (7, 1.2);
\draw[->] (5.1, 1.3) -- node [above, scale = 0.7] {$i$} (7, 1.3);
\draw[->] (5.1, 0.1) -- node [above, sloped, scale = 0.7, near end] {$i$} (7, 1.2);
\draw[->] (5.1, 0) -- node [above, scale = 0.7] {$i$} (7, 0);
\end{tikzpicture}
\]
Here we use the convention that $a$ refers to the circle that involves the left-most arc of the braid closure. As a homomorphism, $a$ stands for left multiplication by $a$ in $\Lambda^\ast(a,b)$, and similarly $b$ stands for left multiplication by $b$. The homomorphism $i\colon \Lambda^\ast(a,b) \to \Lambda^\ast(a)$ is the quotient map which sends both $a$ and $b$ to $a$.

We note that the two $A.3$ morphisms first merge the two circles, and then split the circle using a right-pointing arrow in the lower half of the smoothing. After rotating this arrow counter-clockwise, we see that in one situation it points to the $a$-circle, and in the other situation to the $b$-circle.

At the moment we are on track to calculate the unreduced odd Khovanov homology, but we can pass to a reduced calculation by only considering the subcomplex where the generators contain $a$, and shifting the $q$-degree by $1$. This corresponds to putting the basepoint on the left-most arc in the braid closure. Since we want to calculate $\rKh{\ast,34}(T(3,14))$, the complex reduces to
\begin{equation} \label{eq:almostdone}
\begin{tikzpicture}[baseline={([yshift=-.5ex]current bounding box.center)}]
\node at (0,0) {$\Z\langle a\rangle$};
\node at (0,0.6) {$\oplus$};
\node at (0,1.2) {$\Z\langle a\rangle$};
\node at (4,0) {$\Z\langle a\wedge b\rangle$};
\node at (4,0.6) {$\oplus$};
\node at (4,1.2) {$\Z\langle a\wedge b\rangle$};
\draw[->] (0.5, 1.1) -- node [above, sloped, scale = 0.7, near end] {$-b$} (3.2,0.1);
\draw[-, line width=6pt, color = white] (0.5, 0.1) -- (3.2, 1.1);
\draw[->] (0.5, 1.2) -- node [above, scale = 0.7] {$-2b$} (3.2, 1.2);
\draw[->] (0.5, 0.1) -- node [above, sloped, scale = 0.7, near end] {$b$} (3.2, 1.1);
\draw[->] (0.5, 0) -- node [above, scale = 0.7] {$2b$} (3.2, 0);
\end{tikzpicture}
\end{equation}
where $\Z\langle w\rangle$ stands for the infinite cyclic abelian subgroup of $\Lambda^\ast(a,b)$ generated by the word $w$. Notice that the complex reduces to $0$ in homological degree $6$ and higher, and in fact also in homological degree $3$ and lower. As multiplication by $a$ is $0$ from $\Z\langle a\rangle \to \Z\langle a\wedge b\rangle$, we removed it from the diagram. Also, multiplication by $b$ is an isomorphism, and so we can perform a Gaussian elimination on one of the diagonal homomorphisms. 

If we choose the one pointing downward, the resulting complex is
\[
\begin{tikzpicture}
\node at (0,0) {$\Z\langle a\rangle$};
\node at (3,0) {$\Z\langle a\wedge b\rangle$};
\draw[->] (0.5,0) -- node [above, scale = 0.7] {$-b + 4b$} (2.2,0);
\end{tikzpicture}
\]
which results exactly in the $3$-torsion observed in Figure \ref{fig:torus314}.

Observe also that using $E_2$ instead of $E_1$ would have resulted in the the two horizontal lines in (\ref{eq:almostdone}) to be $0$. In that case we can perform two Gaussian eliminations, and this explains why $\rKh{\ast, 40}(T(3,14)) = 0$. 
\end{example}

\section{Concordance invariants}

As was already observed in \cite{MR3071132} the definition of Rasmussen's $s$-invariant \cite{MR2729272} does not carry over directly to the odd theory.  Indeed, Shumakovitch \cite{MR2777025} found several knots which are \em zero-omitting\em, that is, they have no rational homology in homological grading $0$. Getting an odd version of the $s$-invariant would therefore require a rather different approach.

However, Sarkar, Scaduto, and Stoffregen \cite[\S 5.6]{MR4078823} construct refinements of the Rasmussen-invariant $s_\F$, where $\F=\Z/2\Z$ is the field with two elements, based on on their construction of a stable homotopy type for odd Khovanov homology, and the Steenrod square cohomology operations arising this way. In particular, there is a refinement corresponding to the first Steenrod square $\Sq^1$, which can be computed using the Bockstein homomorphism obtained through the change of coefficients
\[
0\longrightarrow \Z/2\Z \longrightarrow \Z/4\Z \longrightarrow \Z/2\Z \longrightarrow 0.
\] 
Since the even and odd Khovanov homologies agree with $\Z/2\Z$ coefficients, let us simply write $\Kh^{\ast,\ast}(K;\F)$ in this case. Also, since rationally even and odd Khovanov homologies are rather different, we expect the Bockstein homomorphisms on $\Kh^{\ast,\ast}(K;\F)$, depending on whether we use the even or odd theory, to be different in general as well. 
Indeed, for alternating knots the integral odd homology is torsion free, while even (unreduced) homology contains $\Z/2\Z$-summands, so that the Bockstein homomorphisms are different. But we are more interested in examples where the odd theory leads to a non-trivial Bockstein homomorphism that can be picked up by the Sarkar--Scaduto--Stoffregen refinements. In fact, the zero-omitting examples of Shumakovitch appear to be good candidates for this.

Let us recall the definition of the refinements in \cite{MR4078823}, which go back to the work of Lipshitz and Sarkar \cite{MR3189434}. We will focus on the case of Bockstein homomorphisms $\beta\colon \Kh^{i,j}(K;\F)\to \Kh^{i+1,j}(K;\F)$ only, that correspond to the odd and even first Steenrod squares.

The {\em Bar-Natan} complex $C^\ast_{\mathrm{BN}}(\mathcal{D};\F)$ is a deformation of the Khovanov complex $CO^\ast(\mathcal{D};\F)$ (since $\F$ is the field of two elements, we might as well use the odd complex), whose coboundary is given by $\delta_{\mathrm{BN}} = \delta + \gamma$, with $\gamma$ defined as follows. For an edge $e_c$ in the hypercube between vertices $c^0$ and $c^1$ we get homomorphisms
\[
E_{e_c} \colon \Lambda^\ast(S_{c^0};\F) \to \Lambda^\ast(S_{c^1};\F)
\]
which in the case of a merger is induced by $s_0\wedge s_1 \mapsto s$, if the components $s_0$ and $s_1$ merge to $s$, and is $0$ on all generators not having both $s_0$ and $s_1$ in their word. In the case of a split from $s$ to $s_0$ and $s_1$, the map is induced by $s\mapsto s_0\wedge s_1$, and being $0$ on generators not having $s$ in their word.

The coboundary no longer preserves the $q$-grading, but the $q$-grading determines a descending filtration $\mathcal{F}$ on $C^\ast_{\mathrm{BN}}(\mathcal{D};\F)$ with $\mathcal{F}_j/ \mathcal{F}_{j+2}$ the Khovanov complex of $\mathcal{D}$ with coefficients in $\F$.

As was shown in \cite{MR3189434}, the cohomology of the Bar-Natan complex for a knot diagram is concentrated in degree $0$, with $H^0_{\mathrm{BN}}(K) = \F^2$, and the \em Rasmussen invariant \em with coefficients in $\F$ is given by
\begin{align*}
s_\F(K) &= \max\{j\in 2\Z+1\,|\, i^\ast \colon H^0(\mathcal{F}_j) \to \BN^0(K;\F) \mbox{ surjective}\}+1\\
&= \max\{j\in 2\Z+1\,|\, i^\ast \colon H^0(\mathcal{F}_j) \to \BN^0 (K;\F) \mbox{ non-zero}\}-1.
\end{align*}
Denote $p\colon H^0(\mathcal{F}_j;\F) \to \Kh^{0,j}(K;\F) \cong H^0(\mathcal{F}_j/\mathcal{F}_{j+2};\F)$, and consider the following configurations
\begin{equation}\label{eq:fullconfigs}
\begin{tikzpicture}[baseline={([yshift=-.5ex]current bounding box.center)}]
\node at (0,0) {$\langle \tilde{a},\tilde{b}\rangle$};
\node at (3,0) {$\langle \hat{a}, \hat{b} \rangle$};
\node at (6,0) {$\langle a,b\rangle$};
\node at (9,0) {$\langle \bar{a},\bar{b}\rangle$};
\node at (0,1) {$\Kh^{-1,j}(K;\F)$};
\node at (3,1) {$\Kh^{0,j}(K;\F)$};
\node at (6,1) {$H^0(\mathcal{F}_j;\F)$};
\node at (9,1) {$\BN^0(K;\F)$};
\node at (0,2) {$\langle \tilde{a}\rangle$};
\node at (3,2) {$\langle \hat{a}\rangle$};
\node at (6,2) {$\langle a\rangle$};
\node at (9,2) {$\langle \bar{a}\rangle\not=0$};
\draw[right hook->] (0,0.25) -- (0,0.75);
\draw[right hook->] (3,0.25) -- (3,0.75);
\draw[right hook->] (6,0.25) -- (6,0.75);
\draw[-] (9.05,0.25) -- (9.05,0.75);
\draw[-] (8.95,0.25) -- (8.95,0.75);
\draw[right hook->] (0,1.75) -- (0,1.25);
\draw[right hook->] (3,1.75) -- (3,1.25);
\draw[right hook->] (6,1.75) -- (6,1.25);
\draw[right hook->] (9,1.75) -- (9,1.25);
\draw[->] (0.5,0) -- (2.5,0);
\draw[<-] (3.5,0) -- (5.5,0);
\draw[->] (6.5,0) -- (8.5,0);
\draw[->] (0.4,2) -- (2.6,2);
\draw[<-] (3.4,2) -- (5.6,2);
\draw[->] (6.4,2) -- (8.3,2);
\draw[->] (1.1,1) -- node[above] {$\beta$} (2.1,1);
\draw[<-] (3.95,1) -- node[above] {$p$} (5.15,1);
\draw[->] (6.85,1) -- node[above] {$i^\ast$} (8.1,1);
\end{tikzpicture}
\end{equation}

\begin{definition}
Call an odd integer $j$ \em $\beta$-half-full\em, if there exist $a\in H^0(\mathcal{F}_j;\F)$ and $\tilde{a}\in Kh^{-1,j}(K;\F)$ such that $p(a)=\beta(\tilde{a})$, and such that $i^\ast(a)=\bar{a}\not=0$. That is, there exists a configuration as in the upper two rows of (\ref{eq:fullconfigs}).

Call an odd integer $j$ \em $\beta$-full\em, if there exist $a,b\in H^0(\mathcal{F}_j;\F)$ and $\tilde{a},\tilde{b}\in Kh^{-1,q}(K;\F)$ such that $p(a)=\beta(\tilde{a})$, $p(b)=\beta(\tilde{b})$, and $i^\ast(a),i^\ast(b)$ generate $\BN^0(K;\F)$. That is, there exists a configuration as in the lower two rows of (\ref{eq:fullconfigs}).
\end{definition}

We note that while $i^\ast(a)$ and $i^\ast(b)$ have to be non-zero, it is allowed that $p(a)$ or $p(b)$ are zero.

\begin{definition}
Let $K$ be a knot and $\beta\colon \Kh^{i,j}(K;\F)\to \Kh^{i+1,j}(K;\F)$ the first Steenrod Square $\Sqo{e}$ or $\Sqo{o}$ coming from even or odd Khovanov homology, then $r^\beta_+,r^\beta_-, s^\beta_+, s^\beta_-\in \Z$ are defined as follows.
\begin{align*}
r^\beta_+(K) &= \max\{j\in 2\Z+1\,|\, j \mbox{ is $\beta$-half-full}\}+1\\
s^\beta_+(K) &= \max\{j\in 2\Z+1\,|\, j \mbox{ is $\beta$-full}\}+3.
\end{align*}
If $\overline{K}$ denotes the mirror of $K$, we also set
\begin{align*}
r^\beta_-(K) &= -r^\beta_+(\overline{K})\\
s^\beta_-(K) &= -s^\beta_+(\overline{K}).
\end{align*}
We also write
\[
 s^\beta(K) = (r^\beta_+(K),s^\beta_+(K),r^\beta_-(K),s^\beta_-(K)).
\]
\end{definition}
 It is shown in \cite[Lm.4.2]{MR3189434} that $r^\beta_+(K), s^\beta_+(K)\in \{s_\F(K), s_\F(K)+2\}$. Furthermore, $|r^\beta_\pm(K)|/2$ and $|s^\beta_\pm(K)|/2$ are concordance invariants and lower bounds for the smooth slice genus $g_4(K)$, see \cite[Thm.1]{MR3189434} in the even case, and \cite[Thm.5.14]{MR4078823} in the odd case.
 
 \begin{example}
The knot $K = 10_{132}$ is the only knot with at most 10 crossings which is zero-omitting, compare \cite[\S 4.3]{MR2777025}. It is known that $s=s_\F(K) = -2$ and the integral odd Khovanov homology satisfies $\Kho{0,s\pm 1}(K) = \Z/2\Z$ and $\Kho{1,s\pm 1}(K) = 0$. This implies that $\Sqo{o}\colon\Kh^{-1,s\pm 1}(K;\F) \to \Kh^{0,s\pm 1}(K;\F)$ is surjective, making $-1$ $\Sqo{o}$-half-full and $-3$ $\Sqo{o}$-full. Since $\Sqo{o}\colon\Kh^{-1,s\pm 1}(\overline{K};\F) \to \Kh^{0,s\pm 1}(\overline{K};\F)$ is $0$ we get
\[
s^{\Sqo{o}}(10_{132}) = (0, 0, -2, -2).
\]
Notice that this does not give a better bound on the slice genus than $s_\F$.

The knot $K = 11n38$ is the only zero-omitting knot with exactly 11 crossings. We have $s_\F(K) = 0$ and $\Kho{0,-1}(K) = \Z/2\Z\oplus \Z/3\Z$, $\Kho{1,-1}(K) = 0$. As above this implies $-1$ is $\Sqo{o}$-full, and $s_+^{\Sqo{o}}(11n38) = 2$. However, since $\Kho{0,1}(K) = \Z/2\Z\oplus \Z/3\Z$, $\Kho{1,1}(K) = \Z/2\Z$ it is not immediately clear whether $1$ is $\Sqo{o}$-half-full. But this time we do get a better bound on the slice genus.

That $11n38$ is not slice can also be derived from the signature, which is non-zero. In fact, $g_4(11n38) = 1$, see \cite{knotinfo}.

The knot $K = 9_{42}$ is not zero-omitting, but the rational generator appears in the wrong $q$-grading. We have $s_\F(K) = 0$, and $\Kho{1,1}(K) = \Z\oplus \Z/2\Z$, $\Kho{0,1}(K) = 0$. So here we look at the mirror and conclude $s^{\Sqo{o}}_-(K) = -2$.
 \end{example}
 
 In \cite[\S 6]{schuetz2018fast}, we developed an algorithm to calculate $s^{\Sqo{e}}(K)$ by applying Bar-Natan's scanning algorithm directly on the Bar-Natan complex. This algorithm is also implemented in \verb+KnotJob+. The algorithm is readily adapted to the odd situation. One simply runs the scanning algorithm on the odd complex with coefficients in $\Z/4\Z$, and performs the same steps on the Bar-Natan complex with coefficients in $\F=\Z/2\Z$. 
 After the scanning on the odd complex is finished and resulted in a $\Z/4\Z$-bi-complex $C^{\ast,\ast}$, we have a cochain complex $D^\ast$ with filtration that carries the same information as the Bar-Natan complex, and so that $\mathcal{F}_jD^\ast / \mathcal{F}_{j+2} D^\ast= C^{\ast, j}\otimes_{\Z/4\Z}\Z/2\Z$. As in \cite[\S 6]{schuetz2018fast}, we can then calculate $s^{\Sqo{o}}(K)$.
 
 This has also been implemented in \verb+KnotJob+, and we list a few of our calculations here. 
 \begin{table}[ht]
 \begin{tabular}{|l|c|c||l|c|c|}
 \hline
 Knot & $s^{\Sqo{o}}$ & $s_\F$ &   Knot & $s^{\Sqo{o}}$ & $s_\F$  \\
 \hline
 \hline
 $9_{42}$ & $(0,0,-2,-2)$ & $0$ & $10_{132}$ & $(0,0,-2,-2)$ & $-2$ \\
 \hline
 $10_{136}$ & $(0,0,-2,-2)$ & $0$ & $11n12$ & $(2,2,0,0)$ & $2$ \\
 \hline
 $11n19$ & $(-2,-2,-4,-4)$ & $-2$ & $11n20$ & $(0,0,-2,-2)$ & $0$ \\
 \hline
 $11n24$ & $(2,2,0,0)$ & $0$ & $11n38$ & $(2,2,0,0)$ & $0$ \\
 \hline
 $11n70$ & $(4,4,2,2)$ & $2$ & $11n79$ & $(2,2,0,0)$ & $0$ \\
 \hline
 $11n92$ & $(0,0,-2,-2)$ & $0$ & $11n96$ & $(2,2,0,0)$ & $0$ \\
 \hline
 $11n138$ & $(2,2,0,0)$ & $0$ & & &  \\
 \hline
 \end{tabular}
\caption{Prime knots with non-constant $s^{\Sqo{o}}$ and at most $11$ crossings.}
\end{table}
There are also $49$ knots among the $888$ non-alternating $12$-crossing prime knots, $286$ among the $5,110$ non-alternating $13$-crossing knots, and $1,718$ among the $27,436$ non-alternating $14$-crossing prime knots, for which $s^{\Sqo{o}}$ is non-constant.

We note that for all such knots up to $12$ crossings the signature is always different from the $s$-invariant, as observed via \cite{knotinfo}. While this is still true for many of the $13$-crossing knots with non-constant $s^{\Sqo{o}}$, there are exceptions. For example, 
\[
s^{\Sqo{o}}(13n158) = (0,0,-2,-2), 
\]
while both signature and $s$-invariant are $0$ for this knot.

In all of our calculations we observed that $s^{\Sqo{o}}_+ = r^{\Sqo{o}}_+$. This may not be so surprising since the splitting of unreduced odd Khovanov homology also leads to a splitting of $\Sqo{o}$. In \cite{schuetz2018fast} we have made the same observation for $s^{\Sqo{e}}$, even though $\Sqo{e}$ does not split on unreduced even Khovanov homology. Of course, $s^{\Sqo{e}}$ is much less often non-constant than $s^{\Sqo{o}}$.

\bibliography{KnotHomology}
\bibliographystyle{amsalpha}

\end{document}